\documentclass[12pt]{article}
\usepackage{amsmath,amsthm}
\usepackage{latexsym}
\usepackage{cite} 
\usepackage[utopia]{mathdesign}
\usepackage{float}
\floatstyle{boxed} 
\restylefloat{figure}
\usepackage{hyperref}
\usepackage{mathrsfs}
\usepackage{tikz}
\usepackage[all]{xy}
\usepackage{graphicx}
\usepackage{multicol}
\usepackage{mathrsfs}
\usepackage{pst-grad} 
\usepackage{pst-plot}
\usepackage[latin1]{inputenc}
\usepackage{epsfig}
\usepackage{pst-grad} 
\usepackage{pst-plot} 
\usepackage[space]{grffile} 
\usepackage{etoolbox}
\usepackage{epstopdf}

\usepackage{color}

\usepackage{psfrag}
\usepackage{graphics}
\usepackage{graphicx}
\usepackage{amsthm}

\newcommand{\Dowl}{E_+^{uG}}
\newcommand{\wh}{\widehat}

\newcommand{\hro}{\bar{\rho}}
\newcommand{\OO}{\mathscr{O}}

\newcommand{\etb}{\bar{\eta}}

\newcommand{\Sp}{\mathbf{Sp}}

\newcommand{\coi}{\la-1\ra}

\newcommand{\A}{\mathbb {A}}\newcommand{\VV}{\mathbb {V}}

\newcommand{\la}{\langle}
\newcommand{\ra}{\rangle}

\newtheorem{theo}{\bf Theorem}
\newtheorem{defi}{\bf Definition}

\newtheorem{cor}{\bf Corollary}
\newtheorem{prop}{\bf Proposition}
\newtheorem{rem}{\bf Remark}
\newtheorem{ex}{\bf Example}

\newcommand{\Mob}{\mathrm{M\ddot{o}b}}

\newcommand{\Gr}{\mathscr{G}}
\newcommand{\BB}{\mathbb{B}}
\newcommand{\KK}{\mathbb{K}}
\newcommand{\LL}{\mathbb{L}}

\newcommand{\Sy}{\mathbb{S}}

\newcommand{\vect}{\mathrm{Vec}_{\mathbb{K}}}

\newcommand{\spe}{\mathbf{Sp}}

\newcommand{\be}{\begin{equation}}
\newcommand{\eeq}{\end{equation}}

\newcommand{\tiltau}{\widetilde{\tau}}

\begin{document}

\title{On the combinatorics of Riordan arrays and Sheffer polynomials: monoids, operads and monops.}

\author{Miguel A. M\'endez and Rafael S\'anchez}
\date{}
\maketitle
\begin{abstract}We introduce a new algebraic construction, {\em monop}, that combines monoids (with respect to the product of species), and operads (monoids with respect to the substitution of species) in the same algebraic structure. By the use of properties of cancellative set-monops we construct a family of partially ordered sets whose prototypical  examples are the Dowling lattices. They generalize the partition posets associated to a cancellative operad, and the subset posets associated to a cancellative monoid.  Their generalized Withney numbers of the first and second kind are the entries of a Riordan matrix and its inverse. Equivalently, they are the connecting coefficients of two umbral inverse Sheffer sequences with the family of powers $\{x^n\}_{n=0}^{\infty}$. We study algebraic monops, their associated algebras and the free monop-algebras, as part of a program in progress to develop a theory of Koszul duality for monops.
\end{abstract}
{\em Dedicated to the memory of Gian Carlo Rota, 1932-1999.}\\
\section{Introduction}

The systematic study of the Sheffer families of polynomials and of its particular instances: the Appel familes and the families of binomial type, was carried out by G.- C. Rota and his collaborators in what is called the Umbral Calculus (see \cite{MullinRota},
\cite{Rotako, Roman-Rota, Roman}). A Sheffer sequence is uniquely associated to a pair of exponential formal power series, $(F(x),G(x)),$ $F(x)$ invertible with respect to the product of series, and $G(x)$ with respect to the substitution. The Sheffer sequences come in pairs, one is called the umbral inverse of the other. If one Sheffer sequence is associated to the pair $(F(x),G(x))$, its umbral inverse  is associated to the pair $(\frac{1}{F(H(x))},H(x))$, where $H(x)=F^{\langle -1\rangle}(x)$, the substitutional inverse of $F(x)$.   Shapiro et al. introduced in \cite{Shapiro} the  Riordan group of matrices, whose entries in the exponential case, connect two Sheffer families of polynomials. Since that a great number of enumerative applications have been found by these methods. See for example the list of Riordan arrays of $\mathrm{OEIS}$.
 
  The initial motivation of the present research was to find a combinatorial explanation of the inversion process in the group of Riordan matrices. Equivalently, to the Sheffer sequences of polynomials and their umbral inverses. The key tool for such explanation is in the first place the  concept of M\"obius function and M\"obius inversion over partially ordered sets (posets) \cite{Rotamob}. 
  
   For the particular case of families of binomial type, the combinatorics of the process of inversion is related to families of posets of enriched partitions (assemblies of structures). One of the families of binomial type obtained by summation over the poset, and its umbral inverse by M\"obius inversion. Particular cases of them were studied in \cite{Reiner}, \cite{JoniRS} and \cite{Sagan}. The general explanation was found in \cite{Mend-Yang}, where the construction of those posets is based on some special kind of set operads, called $c$-operads.  
 
 A similar approach can be applied to the Appel families. The central combinatorial object in this case is that of a $c$-monoid. A $c$-monoid is a special kind of monoid in the monoidal category of species with respect to the product (See \cite{Joy1}, \cite{Mendlib}. See also  \cite{Marcelol}  for an extensive treatment of monoids and Hopf monoids). Given a $c$-monoid, through its product we are able to build a family of partially ordered sets. For each of these monoids, one  Appel family is obtained by summation over those posets, and its umbral inverse by M\"obius inversion. 
 
 In this article we introduce a new algebraic structure, that we called monop, because it is an interesting mix between monoids and operads. Our first step was to construct a monoidal category, the semidirect product (in the sense of Fuller \cite{Fuller}) of the monoidal categories of species with respect to the product and the positive species with respect to the substitution. Then, we define a monop to be a monoid in such category. From the commutative diagrams satisfied for this kind of monoids we deduce all the main properties of monops. We also introduce the $c$-monops. From a $c$-monop we give a general construction of posets that give combinatorial explanation of the inverses of Riordan matrices by means of M\"obius inversion. Or, equivalently, to Sheffer families and their umbral inverses. We present a number of examples of Appel, binomial and general Sheffer families together with the posets constructed using the present theory. Remarkably, we obtain a new operad, that we call the Dowling operad, which we complemented here to a monop in order to give a construction to the classical Dowling lattice and introduce  $r$-generalizations for $r$ a positive integer. With similar techniques we can define monops on rigid species (species over totally ordered sets), with the operations of ordinal product and substitution. In this way giving combinatorial interpretations to the inversion in the Riordan group associated to pairs of ordinary series $(f(x),g(x))$. In a forthcoming paper we shall deal with the applications of the present theory to the ordinary Riordan matrices.  \\ 
      
    Monops have an independent algebraic interest beyond the enumerative applications given here.  B. Vallet \cite{Valletteposet} proved that, under reasonable conditions, posets associate to a $c$-operad are Cohn-Macaulay \cite{Bjorner, Wachs} if and only if the $c$-operad is Koszul \cite{G-K}. In the same vein of Vallet approach, one of us has proved  \cite{Mend-Kosz} that a $c$-monoid is Koszul if and only if the family of associated posets is Cohn-Macaulay. Our next step in this program shall be the development of a Koszul duality theory for monops. Monoids are closely related to associative algebras. Given a monoid $M$, the analytic functor associated to it (\cite{Joy2}) evaluated in  a vector space is an associative algebra. Then, Koszul duality for monops would establish a deep link between Koszul duality for operads and for associative algebras. And also, interesting connection with the Cohn-Macaulay property for the associated posets and Koszulness of the corresponding monop, unifying in this way the criteria established in \cite{Valletteposet} and in \cite{Mend-Kosz}.

\section{Formal power series}
The exponential generating series (or function) of a sequence of  numbers $f_n$, $n=0,1,\dots$ is the formal power series
\begin{equation*}
F(x)=\sum_{n=0}^{\infty}f_n \frac{x^n}{n!}
\end{equation*}
\noindent  The coefficient $f_n$ will be denoted as $F[n]$, $F[n]:=f_n$. The series $F(x)$ will be called a {\em delta series} if $F[0]=0$ and $F[1]\neq 0$.
For an exponential series $F(x)$ with zero constant term, $F[0]=0$, we denote by $\gamma_k(F)(x)$ its divided power
\begin{equation*}
\gamma_k(F)(x):=\frac{F^k(x)}{k!}.
\end{equation*} 
\noindent The substitution of such a formal power series $F(x)$ in another arbitrary formal power series $G(x)$ is equal to
\begin{equation}
G(F(x)):=\sum_{k=0}^{\infty}G[k]\times\gamma_k(F)(x).
\end{equation}

\begin{defi}\normalfont
	A pair of exponential formal power series $(F(x),G(x))$ is called {\em admissible} if $G[0]=0$. An admissible pair is called a {\em Riordan pair} if $F[0]\neq 0$ and $G(x)$ is a delta series.
\end{defi}
Riordan product of admissible pairs is defined as follows 
\begin{equation}
(F_1(x), G_1(x))\ast(F_2(x),G_2(x)):=(F_1(x).F_2(G_1(x)),G_2(G_1(x)).
\end{equation}
\noindent Admissible pairs of series in $\mathbb{C}[[x]]$ form a monoid with respect to the product $\ast$, having $(1,x)$ as identity.
 The Riordan pairs  form a group, the inverse of $(F(x),G(x))$ given by
\begin{equation}
(F(x),G(x))^{-1}=(F^{-1}(G^{\la -1\ra}(x)),G^{\la -1\ra}(x)).
\end{equation}
\noindent Where $F^{-1}(x)$ and $G^{\la -1\ra}(x)$ denote the multiplicative and substitutional inverses of $F(x)$ and $G(x)$ respectively.
\begin{defi}\normalfont
To an  admissible pair $(F(x),G(x))$ we associate the infinite lower triangular matrix having as entries
	\begin{equation}
	C_{n,k}=H_{k}[n],\; 0\leq k\leq n.
	\end{equation}
	$H_k(x)$ being the series $F(x).\gamma_k(G(x))$. That matrix is denoted as $\langle F(x),G(x)\rangle$. The Riordan product is transported to matrix product by the bracket operator,
 We have that (see \cite{Shapiro}). \begin{equation}\label{fundamental}
\langle G_1(x),F_1(x)\rangle\langle G_2(x),F_2(x)\rangle=\langle G_1(x).G_2(F_2(x)),F_2(F_1(x)).\rangle
\end{equation}
The matrix $\langle F(x),G(x)\rangle$ is called a {\em Riordan array} when $(F(x),G(x))$ is a Riordan pair. Riordan arrays  with  the operation of matrix product form a group that is isomorphic to the group of Riordan pairs. The inverse of the matrix $\langle F(x),G(x)\rangle$ is equal to \begin{equation}\label{inversematrix}\langle F(x),G(x)\rangle^{-1}=\langle F^{-1}(G^{\coi}(x)),G^{\coi}(x)\rangle.\end{equation} 
\end{defi}
\noindent 
The ordinary generating function of the sequence $f_n$ is equal to the formal power series
\begin{equation*}
f(x)=\sum_{n=0}^{\infty}f_n x^n.
\end{equation*}
\noindent We denote by $f[n]$ the nth. coefficient of $f(x)$. 

\begin{defi}\normalfont For an admissible pair $(g(x),f(x))$ of ordinary generating functions we define the associated matrix having as entries the coefficients 
\begin{equation}
C_{n,k}=h_k[n],\;0\leq k\leq n
\end{equation}	
\noindent where $h_k(x)$ is the series $h_k(x)=g(x).f^k(x)$.\end{defi}
	
	\section{Sheffer sequences of polynomials} 
	\begin{defi}\normalfont
		Let $G(x)$ be a delta series. Define the polynomial sequence
		\begin{equation}
		p_{n}(x):=\sum_{k=1}^{n}\gamma_k(G)[n]x^k, \;n\geq 1
		\end{equation}
		\noindent and let $p_0(x)\equiv 1$. This polynomial sequence is known to be of binomial type, 
		$$p_n(x+y)=\sum_{k=0}^n\binom{n}{k}p_{n-k}(x)p_{k}(y).$$
		
		\noindent It is called the {\em conjugate} sequence to the delta series $G(x)$. It is also called the {\em associated} sequence to the series $P(x)=G^{\la -1\ra}(x)$. We have that 
		\begin{equation}
		P(D)p_n(x)=np_{n-1}(x),
		\end{equation}
	\end{defi}
	\noindent where $P(D)$ is the operator defined by 
	$$P(D)=\sum_{n=1}^{\infty}P[n]\frac{D^n}{n!},$$ $D$ being the derivative operator $Dr(x)=r'(x)$.
	
	\begin{defi}\normalfont We say
		that a family of polynomials $s_n(x)$ is {\em Sheffer} if there exists
		 Riordan pair of formal power series  $(F(x), G(x))$ such that
	\begin{equation}\label{riordansheffer}
	s_n(x):=\sum_{k=0}^n (F.\gamma_k(G)[n])x^k,\; n\geq 0.
	\end{equation}

	We will say that $\{s_n(x)\}_{n=0}^{\infty}$ is the {\em conjugate} sequence of $(F(x),G(x))$. \end{defi}
	Observe that the coefficients $c_{n,k}=F.\gamma_k(G)[n]$ connecting the family of powers $x^n$ with $s_n(x)$, $n\geq 0$, are the entries of the Riordan matrix associated to the pair $(F(x),G(x)),$ $\langle F(x),G(x)\rangle$.
	Let us consider the Riordan inverse of  $(F(x),G(x))$,
	 \begin{equation}(S(x),P(x))=(F(x),G(x))^{-1}=(F^{-1}(G^{\coi}(x)), G^{\coi}(x))).\end{equation}\noindent Let $\{p_n(x)\}_{n=0}^{\infty}$ be the family of binomial type associated to the delta operator $P(D)$. It is not difficult to verify that
	\begin{equation}\label{eq:shefferbinomial}
	s_n(x)=S^{-1}(D)p_n(x)=F(G^{\langle -1\rangle}(D))p_n(x)=F(P(D))p_n(x)
	\end{equation}  
	As a consequence of Eq. (\ref{eq:shefferbinomial}), we get that the Sheffer sequence  $\{s_n(x)\}_{n=0}^{\infty}$  satisfies the binomial identity
	$$s_n(x+y)=\sum_{k=0}^n\binom{n}{k} s_k(x)p_{n-k}(y).$$
	We say that it is  Sheffer relative to the binomial family $p_n(x)$. It is called the Sheffer sequence {\em associated} to the Riordan pair $(S(x),P(x))$.
	
	A Sheffer sequence associated to a Riordan pair of the form $(S(x),x)$ is called an Appel sequence. An Appel sequence is  Sheffer relative to the family of powers, $\{x^n\}_{n=0}^{\infty}.$ Observe that, by Eq.(\ref{riordansheffer}), a such Appel sequence $a_n(x)$ conjugate to the pair $(F(x),x)$, $F(x)=S^{-1}(x)$ is of the form,
	\begin{equation}
	a_n(x)=\sum_{k=0}^nF(x).\frac{x^k}{k!}[n]x^k=\sum_{k=0}^n\binom{n}{k}F[n-k]x^k,
	\end{equation}
	since $(F(x)\frac{x^k}{k!})[n]=\binom{n}{k}F[n-k]$. 
	Similarly, a family of binomial type is  Sheffer associated to Riordan pairs of the form $(1,P(x))$ (resp. conjugate to pairs of the form $(1,F(x))$, $F(x)=P^{\langle -1\rangle}(x)$.
	\subsection{Umbral substitution}
	
	Let $$r_n(x)=\sum_{k=0}^n d_{n,k}x^k$$ be another Sheffer sequence conjugated to a Riordan pair $(H(x),K(x))$. Consider the umbral substitution defined by
	\begin{equation*}
	s_n(\mathbf{r})=\sum_{k=0}^n c_{n,k}\mathbf{r}^k:=\sum_{k=0}^n c_{n,k}r_k(x)=\sum_{j=0}^n(\sum_{j\leq k\leq n} c_{n,k}d_{k,j})x^j.
	\end{equation*}
	Since the matrix of coefficients of the umbral substitution  is the product of the corresponding matrices, by Eq. (\ref{fundamental}), we have that
	\begin{prop}\label{prop.umbral}\normalfont
		The umbral substitution $s_n(\mathbf{r})$ of two Sheffer sequences as above is also Sheffer, conjugated to the Riordan product 
		$$(F(x),G(x))\ast (H(x),K(x))=(F(x)H(G(x)),K(G(x))).$$
	\end{prop}
\begin{cor}\normalfont 
	Let $a_n(x)$ and $p_n(x)$ be the Appel and binomial sequences conjugate respectively to $(F(x),x)$ and $(1,G(x))$. Then we have
	\begin{equation}
	s_n(x)=a_n(\mathbf{p}).
	\end{equation}
\end{cor} 
\begin{proof}Immediate from Prop. \ref{prop.umbral} and the identity $$(F(x),G(x))=(F(x),x)\ast(1,G(x)).$$
	\end{proof}
	The Sheffer sequence associated to $(F(x),G(x))$ is the {\em umbral inverse} of $\{s_n(x)\}_{n=0}^{\infty}$, denoted $\{\widehat{s}_n(x)\}_{n=0}^{\infty}$. For every $n\geq 0,$
	\begin{equation}\label{eq:umbral}s_n(\widehat{\mathbf{s}}(x)):=\sum_{k=1}^n F(x).\gamma_k(G(x))[n]\widehat{s}_k(x)=x^n=\widehat{s}_n(\mathbf{s}(x))\end{equation}
	
	\noindent This is obviously equivalent to the identity (\ref{inversematrix}). It says that the matrix $F(x).\gamma_k (G(x))[n]$ is the inverse of $S(x).\gamma_k (P(x))[n]$.  It is summarized in the following table. 
	
	\begin{center}
		\begin{tabular}
			{|c| c| c| c|c|}\hline
			 &Sheffer& Appel & Binomial& Umbral Inverse\\ 
			\hline
			Associated to&(S(x),P(x))& (S(x),x) & (1,P(x))&Conjugate to\\ 
			\hline
			Conjugate to & $(F(x),G(x))$ &$(F(x),x)$ &$(1,G(x))$&Associated to\\\hline
			Matrix&$F(x)\frac{G^k(x)}{k!}[n]$&$\binom{n}{k}F[n-k]$&$\frac{G^k(x)}{k!}[n]$&Inverse Matrix\\ \hline
		\end{tabular}
	\end{center}
  
\section{Species and rigid species}
In a general way, a (symmetric) species is a covariant functor from the category of finite sets and bijections $\mathbb{B}$ to a suitable category. For example, if we set as codomain the category of finite sets and functions $\mathbb{F}$, we get set species (see \cite{BBL, Joy1}). If we instead set as codomain the category of vector spaces and linear maps $\vect$; we get linear species (see for example \cite{Joy2, Marcelol, Mendlib}). By changing the domain $\mathbb{B}$ by the category of totally ordered sets $\mathcal{L}$ and poset isomorphisms, we obtain rigid species (species of structures without the action of the symmetric groups, non-symmetric species). Rigid species are endowed with two kinds of operations; shuffle and ordinal. 
\subsection{Three monoidal categories with species.} The (symmetric) set species, together with the natural transformation between them form a category. A species $P$ is said to be {\em positive} if it assigns no structures to the empty set, $P[\emptyset]=\emptyset$. The category of species will be denoted by $\spe$ and the category of positive species by $\spe_+$.

Recall that the product of species is defined as follows
\begin{equation*}
 (M.N)[V]=\sum_{V_1+V_2=V}M[V_1]\times N[V_2].
 \end{equation*}
And the substitution of a positive species $P$ into an arbitrary species $R$ by
 \begin{equation*}
 R(P)[V]=\sum_{\pi\in\Pi[V]}\prod_{B\in \pi}P[B]\times R[\pi].
 \end{equation*}
 \noindent
 \noindent  The symbol of sum in set theoretical context will always denote disjoint union. The elements of the product $M.N[V]$ are pairs $(m,n)$, $m$ an element of $M[V_1]$ and $n$ an element of $N[V_2]$, for some decomposition of $V$, $V=V_1+V_2.$
 The category $\spe$ is monoidal with respect to the operation of product. It has as identity the species $1$ of empty sets,
 \begin{equation}
 1[V]:=\begin{cases}
 \{\emptyset\}&\mbox{if $V=\emptyset$}\\
\emptyset&\mbox{otherwise},
 \end{cases}
 \end{equation}
\noindent we have canonical isomorphisms $$1.M\cong M\cong M.1.$$

The category of positive species is monoidal with respect to the operation of substitution. Its identity being the species of singletons,
\begin{equation}
X[V]:=\begin{cases}
V&\mbox{if $|V|=1$}\\
\emptyset&\mbox{otherwise}.\end{cases}
\end{equation}
The divided power $\gamma_k(G(x))=\frac{G^k(x)}{k!}$ of an exponential formal power series $G(x)$ has a counterpart in species. Recall that for a positive species $P$,
\begin{equation*}
\gamma_k(P)[V]=\sum_{|\pi|=k}\prod_{B\in \pi}P[B].
\end{equation*} 
The elements of $\gamma_k(P)$ are assemblies of $P$-structures having exactly $k$ elements, $$\mbox{$a=\{p_B\}_{B\in\pi}$, $|a|=k$, and  $p_B\in P[B]$ for every $B\in \pi$.}$$
The elements of the substitution $R(P)$ are pairs of the form: $(a,r)$, $a=\{p_B\}_{B\in\pi}$ an assembly of $P$-structures, and $r$ an element of $R[\pi]$. The divided power can be seen as the substitution of $P$ into the species $E_k$, of sets of cardinal $k$, $$E_k[V]=\begin{cases}
\{V\}&\mbox{if $|V|=k$}\\ \emptyset&\mbox{ otherwise,}\end{cases}$$
 $$\gamma_k(P)=E_k(P).$$

\begin{defi}\normalfont
Let us consider now the product category $\spe\times \spe_+$. A pair of species $(M,\OO)$ in $\spe\times\spe_+$ will be called {\emph admissible.}  Morphisms are pairs of natural transformations of the form $$(\phi,\psi):(M_1,\OO_1)\rightarrow (M_2,\OO_2),$$ $\phi:M_1\rightarrow M_2$, and $\psi:\OO_1\rightarrow \OO_2$. It is a monoidal category with respect to the Riordan product, defined as follows:
\begin{equation}
(M_1,\OO_1)\ast (M_2,\OO_2)=(M_1.M_2(\OO_1),\OO_2(\OO_1))
\end{equation}
\noindent having as identity the pair $(1,X)$,
\begin{equation}
(1,X)\ast (M,\OO)=(1.M(X),\OO(X))\cong(M,\OO)\cong (M.1,X(\OO))=(M,\OO)\ast (1,X).
\end{equation}
\noindent It will be called from now on the {\em Riodan category.} 
\end{defi}

The monoidal categories $\spe$ and $\spe_+$ are respectively imbedded into the Riordan category by mapping,
\begin{eqnarray}
M&\mapsto&(M,X)\\
\OO&\mapsto&(1,\OO).
\end{eqnarray}
\begin{rem}\normalfont The Riordan category is just the semidirectproduct $\spe\rtimes\spe_+$ (in the sense of \cite{Fuller}) associated to the action
\begin{eqnarray*}
\spe_+&\rightarrow&[\spe,\spe]\\
\OO&\mapsto& (M\mapsto M(\OO)).
\end{eqnarray*} 
\end{rem}
\noindent The exponential generating functions of $(M,\OO)$ is defined to be 
	\begin{equation}
	(M,\OO)(x)=(M(x),\OO(x)).
	\end{equation}
	The generating function of the Riordan product $(M_1,\OO_1)\ast(M_2,\OO_2)$ is obviously the Riordan product of the respective generating functions
	\begin{equation*}
	(M_1,\OO_1)\ast(M_2,\OO_2)(x)=(M_1(x).M_2(\OO_1(x)),\OO_2(\OO_1(x))).
	\end{equation*}
\noindent The matrix associated to an admissible pair $(M(x),\OO(x))$ 
  is invertible if and only if $(M(x),\OO(x))$ is a Riordan pair. Since 
  \begin{equation}
  C_{n,k}=|M.\gamma_k(\OO)[n]|
  \end{equation}
  \noindent it enumerates pairs of the form $(m,a)$, $m\in M[V_1]$ and $a$ an assembly of $\OO$-structures over $V_2$ having exactly $k$ elements, $V_1+V_2=[n]$. 
\begin{ex}\normalfont Let us consider $E$, the species of sets, $E[V]=\{V\}$. Let $E_+$ be its associated positive species. 
	The  pair $(E,E_+)$ has as generating function the Riordan pair
	\begin{equation}
	(E,E_+)(x)=(e^x,e^x-1).
	\end{equation}	
	\noindent  The matrix associated to the pair $(e^x,e^x-1)$, $C_{n,k}=|E.\gamma_k(E_+)[n]|$, 
counts the number of partial partitions of $[n]$ having $k$ blocks. The matrix $C_{n,k}=|\Pi.\gamma_k(E_+)[n]|$ associated to the pair $(\Pi,E_+)$
	\begin{equation}
(\Pi,E_+)(x)=(e^{e^x-1},e^x-1)
\end{equation}
 counts pairs of partitions $(\pi_1,\pi_2)$, $\pi_1\in \Pi[V_1]$, $\pi_2\in\Pi[V_2]$, $V_1+V_2=[n]$, $\pi_2$ having exactly $k$ blocks.\end{ex}
With this general interpretation in mind we can give a direct combinatorial proof to Eq. (\ref{fundamental}).  
\begin{prop}\label{matrixproduct}\normalfont
	Let $(M_1, \OO_1)$ and $(M_2,\OO_2)$ be two admissible pairs. The matrix associated to the Riordan product $(M_1(x), \OO_1(x))\ast(M_2(x),\OO_2(x))$ is equal to the product of the respective matrices,
	\begin{equation*}
	\langle M_1(x), \OO_1(x)\rangle\cdot\langle M_2(x),\OO_2(x)\rangle=\langle M_1(x)\cdot M_2(\OO_1(x)),\OO_2(\OO_1(x))\rangle.
	\end{equation*}
\end{prop}
\begin{proof}
	The entries of the Riordan arrays  associated to the pairs  $(M_1(x),\OO_1(x))$ and $(M_2(x),\OO_2(x))$ have respectively  the entries
	\begin{eqnarray*}
A_{n,k}&=&|M_1\gamma_k(\OO_1)[n]|\\
B_{n,k}&=&|M_2\gamma_k(\OO_2)[n]|
	\end{eqnarray*}
	The Riordan product $(M_1(x), \OO_1(x))\ast(M_2(x),\OO_2(x))$ is equal to
	$$(M_1(x).M_2(\OO_1(x)),\OO_2(\OO_1(x))).$$ The $n,k$ entry of the matrix associated to the pair $(M_1.M_2(\OO_1),\OO_2(\OO_1))$ is equal to
	$$C_{n,k}=|M_1.M_2(\OO_1).\gamma_k(\OO_2(\OO_1))[n]|.$$
	By associativity of the substitution of species (since $\gamma_k(\OO_2(\OO_1))=E_k(\OO_2(\OO_1))$), and right distributivity of the operation of substitution  with respect to the product, we have that \begin{equation}
	\label{eq:asomonop}
	M_1.M_2(\OO_1).\gamma_k((\OO_2(\OO_1))=M_1.M_2(\OO_1).\gamma_k(\OO_2)(\OO_1)=M_1.(M_2.\gamma_k(\OO_2))(\OO_1).\end{equation}
The elements of the species in the right hand side of Eq. (\ref{eq:asomonop}) evalauted in the set $[n]$ are of the form $(m_1, a_1,(a_2,m_2)),$ where;
	\begin{itemize}
		\item $m_1$ is an $M_1$-structure over a subset $V_1$ of $[n]$.
		\item $a_1$ is an assembly of $\OO_1$-structures over the set $V_2=[n]-V_1$.
		\item $(m_2,a_2)$ is an element of $(M_2.\gamma_k(\OO_2))[\pi]$, $\pi$ being the partition subjacent to $a_1$. The assembly $a_2$ has $k$ elements.  
	\end{itemize}
Assuming that $|\pi|=j$, the set of elements of above is then equipotent with the set (see Fig. \ref{fig.riordanbijective})
$$M_1.\gamma_j(\OO_1)[n]\times M_2.\gamma_k(\OO_2)[j]$$ 
Since $k\leq |\pi|\leq n$, then
\begin{equation*}
C_{n,k}=|M_1.(M_2.\gamma_k(\OO_2))(\OO_1)[n]|=\sum_{j=k}^n|M_1.\gamma_j(\OO_1)[n]| |M_2.\gamma_k(\OO_2)[j]|=\sum_{j=k}^nA_{n,j}B_{j,k}.
\end{equation*}\end{proof}

\begin{figure}
	\begin{center}
		\includegraphics[width=120mm]{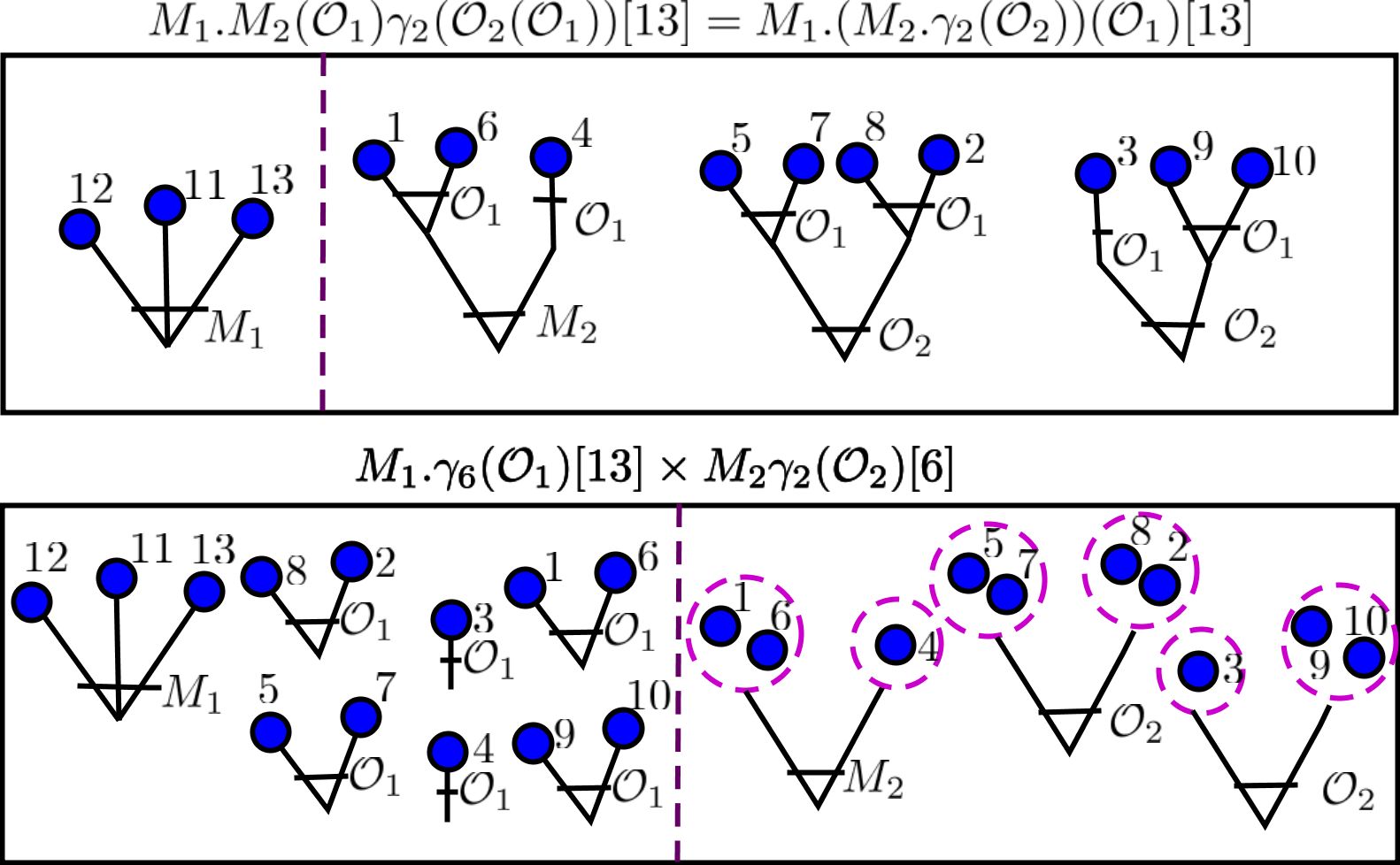}
	\end{center}\caption{The equipotence $\mbox{$M_1.(M_2.\gamma_2(\mathcal{O}_2))(\mathcal{O}_1)[13]\equiv \sum_{j=2}^{13}M_1.\gamma_j(\mathcal{O}_1)[13]\times M_2\gamma_2(\mathcal{O}_2)[j].$}$ }\label{fig.riordanbijective}
\end{figure}
Similar monoidal categories are defined on rigid species. Let  $R, S:\mathcal{L}\rightarrow\mathbb{F}$ two rigid species.  For $l$, a linear order on a set $V$, recall that the shuffle product and substitution are defined respectively by
\begin{eqnarray}
(R.S)[l]&=&\sum_{V_1+V+2=V}R[l_{V_1}]\times S[l_{V_2}]\\\label{eq:shufflesubstitu}
R(S)[l]&=&\sum_{\pi\in\Pi[V]}\prod_{B\in\pi}S[l_B]\times R[\pi] \mbox{  (for $S$ positive).} 
\end{eqnarray} 
\noindent where for $V_1\subseteq V$, $l_{V_1}$ denotes the restriction of the total order $l$ to $V_1$. Note that $l$ induces a total order on any partition of $V$. We say that $B<B'$, for $B,B'\in \pi$ if the minimun element of $l_B$ is smaller in $l$ than the corresponding minimun element of $l_{B'}$. Applications of monops in the context of rigid species with ordinal product and substitution will be consider in a separated paper.

\subsection{Cancellative monoids, cancellative operads, and posets.}
A monoid in the monoidal category $\spe$, the species with  the operation of product, is called (by language abuse) a monoid. An operad is a  monoid in the category $\spe_+$ of positive  species with respect to the substitution.
More specifically. A monoid is a triplet $(M,\nu,\mathfrak{e})$ such that the product $\nu:M\cdot M\rightarrow M$ is associative, and $\mathfrak{e}:1\rightarrow M$, choses the identity, an element of $M[\emptyset]$. We also denote it by $\mathfrak{e}$, by abuse of language. 
We have then the associativity and identity properties
\begin{equation*}
\nu\,(\nu\,(m_1,m_2),m_3)=\nu\,(m_1,\nu\, (m_2,m_3))
\end{equation*}
\begin{equation*}
\nu\,(m,\mathfrak{e})=m=\nu\,(\mathfrak{e},m).
\end{equation*}
\noindent for every triplet of elements $(m_1,m_2,m_3)$ of $M[V_1]\times M[V_2]\times M[V_3]\subseteq (M\cdot M\cdot M)[V]$, and the pairs $(\mathfrak{e},m)$ and $(m,\mathfrak{e})$ respectively in $M[\emptyset]\times M[V]$ and $M[V]\times M[\emptyset].$
\noindent A monoid $(M,\nu,\mathfrak{e})$ (in $\spe$) is called a $c$-monoid if 
\begin{enumerate}
	\item $|M[\emptyset]|=1$
	\item The product $\nu$ satisfies the left cancellation law
	$$\nu\,(m_1,m_2)=\nu\,(m_1,m_2')\Rightarrow m_2=m_2'.$$
\end{enumerate}
And operad, as a monoid in $\spe$ consists of a triplet $(\OO,\eta,\nu)$, where the product
$\eta:\OO(\OO)\rightarrow\OO$ is associative, and  for each unitary set $\{v\}$, $e:X\rightarrow \OO$ chooses the identity in $\OO[\{v\}]$, denoted by $e_{v}$. The product $\eta$ sends pairs of the form $(\{\omega_B\}_{B\in\pi},\omega_{\pi})$ into a bigger structure, $\omega_{V}=\eta(\{\omega_B\}_{B\in\pi},\omega_{\pi})$. Intuitively this product can be thought of as if $\eta$ would assemble the pieces in  $a=\{\omega_B\}_{B\in\pi}$ according to the external structure $\omega_{\pi}$.
 Associativity reads as follows,
\begin{equation*}
\eta(\bar{\eta}(a_1,a_2), \omega_{\pi})=\eta(a_1,\eta(\widehat{a}_2,\omega_{\pi})),
\end{equation*}
\noindent where $\widehat{a}_2$ is isomorphic to $a_2$.
By simplicity we will usually identify $\widehat{a}_2$ with $a_2$.

The identity property reads as follows
\begin{equation*}
\eta(\{e_{v}\}_{v\in V},\omega_V)=\omega_V=\eta(\{\omega_V\},e_{\{V\}}).
\end{equation*}
See \cite{Mendlib} for details and pictures. All this properties can be expressed by the commutativity of the diagrams of monoids in a monoidal category, see Section \ref{section.diagrams}.

\noindent An operad $(\OO,\eta,e)$ is called a $c$-operad if \begin{enumerate}
	\item $|\OO[1]|=1$
	\item The product $\eta$ satisfies the left cancellation law. For $(a,\omega), (a,\omega')\in \OO(\OO)[V]$, we have
	$$\eta(a, \omega)=\eta(a,\omega')\Rightarrow \omega=\omega'.$$ 
\end{enumerate}

\begin{ex}\label{ex.graphs}\normalfont
	The species of simple (undirected) graphs $\mathscr{G}$ is a $c$-monoid (in the category $\spe$). 
	\begin{eqnarray*}
	\nu_1:\Gr.\Gr&\rightarrow&\Gr\\
	(g_1,g_2)&\mapsto& g_1+g_2.
	\end{eqnarray*}
	\noindent The plus sing meaning the disjoint union of the two graphs. There is another monoidal structure over $\Gr$, the product $\nu_2$ sending a pair of graphs to the graph obtained by connecting with edges all the vertices in $g_1$ with those in $g_2$. The two monoidal structures are isomorphic by the correspondence $c:g\mapsto g^c$, sending a graph to its complement, obtained by taking the complementary set of edges (with respect to the complete graph). The natural transformation $c$ is a monoid involutive isomorphism, $c^2=\mathrm{I}_{\Gr}$. The following diagram commutes (see also Fig. \ref{fig.producstgraph})
	\begin{equation}
	\xymatrix{\Gr.\Gr\ar[d]^{c.c}\ar[r]^{\nu_2}&\Gr\ar[d]^{c}\\ \Gr.\Gr\ar[r]^{\nu_1}&\Gr}
	\end{equation}
	
	\begin{figure}
		\begin{center}
		\includegraphics[width=80mm]{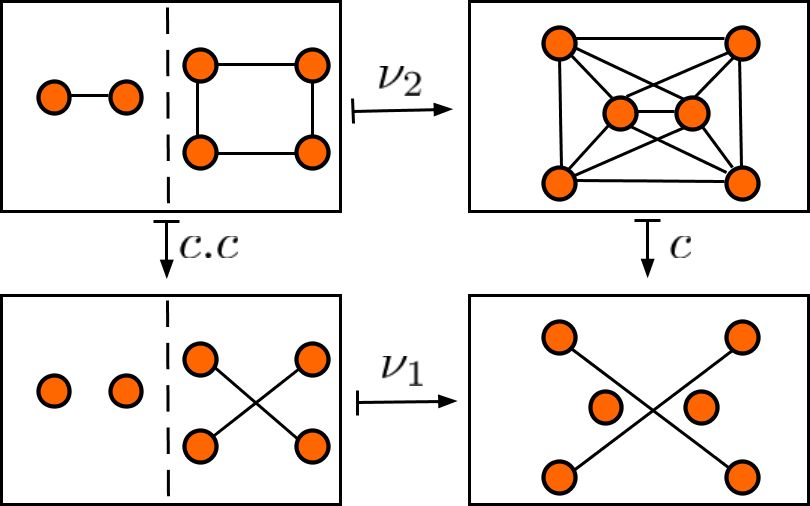}
		\end{center}\caption{The complement isomorphism between two monoidal structures on $\mathscr{G}$}\label{fig.producstgraph}
	\end{figure}
	
	The corresponding positive species $\mathscr{G}_+$ is a $c$-operad  with $\eta(\{g_B\}_{B\in\pi},g_{\pi})=g_V$,  with $g_V$ as the graph obtained by keeping all the edges of the internal graphs plus some more edges created using the information of the external graph $g_{\pi}$.  For each external edge $\{B_{1},B_{2}\}$ 
	of $g_{\pi}$, add all the edges of the form $\{b_{1},b_{2}\}$ 
	with $b_{1}\in B_{1}$ and $ b_{2}\in B_{2}$. The species of connected graphs $\Gr_c$  is a suboperad of $\Gr$. 
\end{ex}

For a $c$-monoid $(M,\nu,\mathfrak{e})$ we define a family of partially ordered sets
\begin{equation*}
P_M[V]=(\uplus_{V_1\subseteq V}M[V_1],\leq_{\nu})=(M.E[V],\leq_{\nu}),\; V\in \BB
\end{equation*} 
\noindent the relation $\leq_{\nu}$ defined by 
\begin{equation*}
m_1\leq_{\nu} m_2,\mbox{ if }\nu\,(m_1,m_2')=m_2
\end{equation*}
\noindent for some $m'_2$. The poset $P_M[V]$ has a zero, the unique element of $M[\emptyset]$. The M\"obius cardinal of $P_M[V]$, $|P_M[V]|_{\mu}$ is defined to be 
\begin{equation*}|P_{M}[V]|_{\mu}=\sum_{m\in M[V]}\mu(\widehat{0},m),
\end{equation*}
\noindent where $\mu$ is the M\"obius function of $P_M[V]$.
In a similar way, for a $c$-operad $(\OO,\eta,e)$ we define a family of posets
\begin{equation*}
P_{\OO}[V]=(E_+(\OO)[V],\leq_{\nu}).
\end{equation*}
\noindent The elements of $E_+(\OO)$ are assemblies of $\OO$-structures. The order relation $\leq_{\nu}$ defined by
\begin{equation*}
a_1\leq a_2 \mbox{ if there exists } \widehat{a}_2 \mbox{ such that }\overline{\eta}(a_1,\widehat{a}_2)=a_2, 
\end{equation*}
\noindent where $\widehat{a}_2$ is an assembly with labels over the partition $\pi_1$ associated to $a_1$, and having $\widehat{\pi}_2$ as associated partition, $\widehat{a}_2=\{\widehat{w}_D\}_{D\in\widehat{\pi}_2}$. The product $\overline{\eta}$  defined as follows \begin{equation*}
\overline{\eta}(a_1,\widehat{a}_2)=\{\eta(\{\omega_C\}_{C\in D},\widehat{w}_D)\}_{D\in\widehat{\pi}_2}.
\end{equation*}
The poset $P_{\OO}[V]$ has a zero, the assembly of singletons $\{e_v\}_{v\in V}$, $e_v$ the unique element of $\OO[\{v\}].$
For $M$ a $c$-monoid and $\OO$ a $c$-operad, we define the M\"obius generating functions of the respective family of posets
\begin{eqnarray*}
	\Mob P_M(x)&=&\sum_{n=0}^{\infty}|P_M[n]|_{\mu}\frac{x^n}{n!}\\
	\Mob P_{\OO}[n]&=&\sum_{n=1}^{\infty}|P_{\OO}[n]|_{\mu}\frac{x^n}{n!}.
\end{eqnarray*}	
\noindent We have that \begin{eqnarray}
\Mob P_M(x)&=&M^{-1}(x)\label{eq.invmonoid}\\\Mob P_{\OO}(x)&=&M^{\coi}(x).\label{eq.inop}
\end{eqnarray}\noindent See \cite{Mend-Yang,  Mendlib}. Moreover, we have.
\begin{prop}\label{prop.umbralinv.ap.bin}\normalfont
	If we define the Appel and binomial families conjugated respectively to $M(x)$ and $\OO(x)$
	\begin{eqnarray}
	\widehat{a}_n(x)&=&\sum_{(m_{V_1},V_2)\in P_M[n]}x^{|V_2|}=\sum_{k=1}^n \binom{n}{k}|M[k]|x^{n-k}\\
	\widehat{p}_n(x)&=&\sum_{a\in P_{\OO}[n]}x^{|a|}=\sum_{k=1}^n|\gamma_k(\OO)[n]|x^k,
	\end{eqnarray}
	\noindent then, we have that their corresponding umbral inverses are obtained by M\"obius inversion over the respective posets
	\begin{eqnarray}a_n(x)&=&\sum_{(m_{V_1},V_2)\in P_M[n]}\mu(\widehat{0},(m_{V_1},V_2))x^{|V_2|}\\
		p_n(x)&=&\sum_{a\in P_{\OO}[n]}\mu(\widehat{0},a)x^{|a|}=\sum_{k=1}^n|\gamma_k(\OO)[n]|_{\mu}x^k.
	\end{eqnarray} 
\end{prop}
\begin{proof}
	A proof of a more general proposition is given in Section \ref{secposetmonop}, Theorem \ref{teo.inversematrix}.
\end{proof}

\subsection{Examples of $c$-Monoids and Appel polynomials} 
\begin{ex}\label{ex.Boolean}\normalfont {\em Pascal matrix, shifted powers.}
	For the monoid $E$, $P_E[n]$ is the Boolean algebra of subsets of $[n]$. The conjugate Appel is  the shifted power sequence $$\sum_{A\subseteq [n]}x^{|A|}=\sum_{k=0}^n\binom{n}{k}x^k=(x+1)^n.$$ The umbral inverse obtained by M\"obius inversion over $P_E[n]$ gives us their  Appel umbral inverse  
		$$\sum_{A\subseteq [n]}\mu(\emptyset,[n]-A)x^{|A|}=\sum_{k=0}^n\binom{n}{k}(-1)^{n-k}x^{k}=(x-1)^n.$$ 
		Consider the power $E^r$, the ballot monoid. The elements of $E^r[V]$ are weak $r$ compositions of $V$, i.e., $r$-uples of pairwise disjoint sets $(V_1,V_2,\dots,V_r)$ (some of them possibly empty) whose union is $V$. It is a c-monoid by adding r-uples component to component:
		
		$$((V_1,V_2,\dots, V_r),(V'_1,V'_2,\dots, V'_r))\stackrel{\nu}{\mapsto}(V_1+V'_1,V_2+V'_2,\dots,V_r+V'_r).$$
		The ballot poset $P_{E^r}[n]$ gives us the combinatorial interpretation of the umbral inversion between the Appel families $(x+r)^n$ and $(x-r)^n$.
\end{ex}
\begin{ex}{\em Euler numbers}\normalfont\label{ex.eulerpolynomials}
	The species of sets of even cardinal, $E^{\mathrm{ev}}$, is a submonoid of $E$. Its generating function is equal to the hyperbolic cosine, $$E^{\mathrm{ev}}(x)=\frac{e^x+e^{-x}}{2}=\cosh(x).$$ It gives us  $P_{E^{\mathrm{ev}}}[n]$, the poset of subsets of $[n]$ having even cardinal. Since $$\mathrm{sech}(x)=E^{\mathrm{ev}}(x)^{-1}=1+\sum_{k=1}^{\infty} (-1)^nE^{*}_n\frac{x^{2n}}{2n!},$$ ($E^*_n$ being Euler or secant numbers, that count the number of zig permutations, OEIS A000364). We have that $$|P_{E^{\mathrm{ev}}}[n]|_{\mu}=\begin{cases}(-1)^{\frac{n}{2}}E_{n/2}^{*}=\mu(\widehat{0}, [n])&\mbox{$n$ even}\\0&\mbox{$n$ odd.}\end{cases}$$ 
	
	 The corresponding conjugate Appel polynomials are (OEIS A119467)
	\begin{equation*}
	\widehat{a}_n(x)=\sum_{A\subseteq [n], |A|\mbox{ even}}x^{n-|A|}=\sum_{k=0}^{\lfloor \frac{n}{2}\rfloor}\binom{n}{2k}x^{n-2k}=\frac{(x+1)^n+(x-1)^n}{2}.
	\end{equation*}
	\noindent and its umbral inverses (OEIS A119879)
	\begin{equation*}
	a_n(x)=\sum_{k=0}^{\lfloor \frac{n}{2}\rfloor}\binom{n}{2k}(-1)^kE^*_{k}x^{n-2k}.
	\end{equation*}
	We have the identity (in umbral notation), $$(\mathbf{a}+1)^n+(\mathbf{a}-1)^n=2x^n.$$
	And, making $E_n=|P_{E^{\mathrm{ev}}}[n]|_{\mu}=a_n(0)$,
	$$(\mathbf{E}+1)^n+(\mathbf{E}-1)^n=2\delta_{n,0}.$$
	The classical Euler polynomials $\mathcal{E}_n(x)$ are connected with $a_n(x)$ by the formulas
	\begin{eqnarray*}
	a_n(x)&=&\sum_{k=0}^n\binom{n}{k} 2^k\mathcal{E}_k(\frac{x}{2})\\
	\mathcal{E}_n(x)&=&\frac{1}{2^n}\sum_{k=0}^n \binom{n}{k}(-1)^k a_{n-k}(2x)
	\end{eqnarray*} 
The first identity follows by manipulating their generating functions, the second by binomial inversion.
\end{ex}
\begin{ex}
	\normalfont {\em Free commutative monoid generated by a positive species.} Let $M$ be a positive species, the free commutative monoid generated by $M$ is $E(M)$, the species of assemblies of $M$-structures. It is a c-monoid with the operation $(a_1, a_2)\stackrel{\nu}{\mapsto}a_1+a_2$, taking the union of pairs of assemblies. The order in $P_{E(M)}[n]$ is given by the subset relation on partial assemblies: $(a_1,V_1)\leq (a_2,V_2)$ if $a_1\subseteq a_2$. Its M\"obius function is \begin{equation}\mu((a_1,V_1),(a_2,V_2))=(-1)^{|a_2-a_1|}.\end{equation}
	The corresponding Appel polynomials are
	\begin{eqnarray}
	\widehat{a}_n(x)&=&\sum_{(a,V)\in P_{E(M)}[n]}x^{|V|}\\
	a_n(x)&=&\sum_{(a,V)\in P_{E(M)}[n]}(-1)^{|a|}x^{|V|}
	\end{eqnarray}
	Subsequent Examples \ref{ex.Hermite}, \ref{ex.Bell-appel}, and \ref{ex.graph.appel} are particular cases of this general construction.
\end{ex}
	\begin{ex}\normalfont\label{ex.Hermite}
		{\em Hermite Polinomials.}  Consider the free commutative monoid generated by $E_2$, the species of sets of cardinal $2$.  It is the species of parings. Equivalently, the species of partitions whose blocks all have cardinal $2$, $E(E_2)$, $$E(E_2)(x)=e^{\frac{x^2}{2}}.$$ The elements of the poset $P_{E(E_2)}[n]$ are partial partitions of $[n]$ having blocks of length two (partial pairings), endowed with the relation $\pi_1\leq \pi_2$ if every block of $\pi_1$ is a block of $\pi_2$.  The signless Hermite polynomials $\widehat{H}_n(x)$, are obtained as a sum over the elements of $P_{E(E_2)}[n]$. Their umbral inverses, the Hermite polynomials $H_n(x)$ are obtained by M\"obius inversion, see Fig. \ref{fig.hermite}. In the figure, partial pairing are   identified with a total partition having blocks of either size one or two. For example, following this convention, the partial partition of pairings  $25|57$ in $\{1,2,3,4,5, 6,7\})$ is represented as a total partition $25|57|1|3|6$ (also represented as the pair  $(25|57, \{1, 3, 6\})$). In the following equations, $\pi$ will represent a partial partition consisting only of pairings. 
		\begin{eqnarray*}
		\widehat{H}_n(x)&=&\sum_{\pi\in P_{E(E_2)}[n]} x^{n-2|\pi|}
		=\sum_{0\leq k\leq \lfloor\frac{n}{2}\rfloor}\binom{n}{2k}\frac{(2k)!}{k!2^k}x^{n-2k}\\&=&\sum_{0\leq k\leq \lfloor\frac{n}{2}\rfloor}\binom{n}{2k}(2k-1)!!x^{n-2k} \\
		H_n(x)&=&\sum_{\pi\in P_{E(E_2)}[n]}\mu(\widehat{0},\pi)x^{n-|\pi|}=\sum_{\pi\in P_{E(E_2)}[n]}(-1)^{|\pi|}x^{n-2|\pi|}\\
		&=&\sum_{0\leq k\leq \lfloor\frac{n}{2}\rfloor}\binom{n}{2k}(-1)^k(2k-1)!!x^{n-2k}.
		\end{eqnarray*}
		This elementary M\"obius inversion is closely related to Rota-Wallstrom combinatorial approach to stochastic integrals  for the case of a  totally random Gaussian measure. See \cite{Rota-Wallstrom}, and \cite{Pecatti}. 
	\end{ex}
	\begin{figure}
\begin{center}
	\includegraphics[width=157mm]{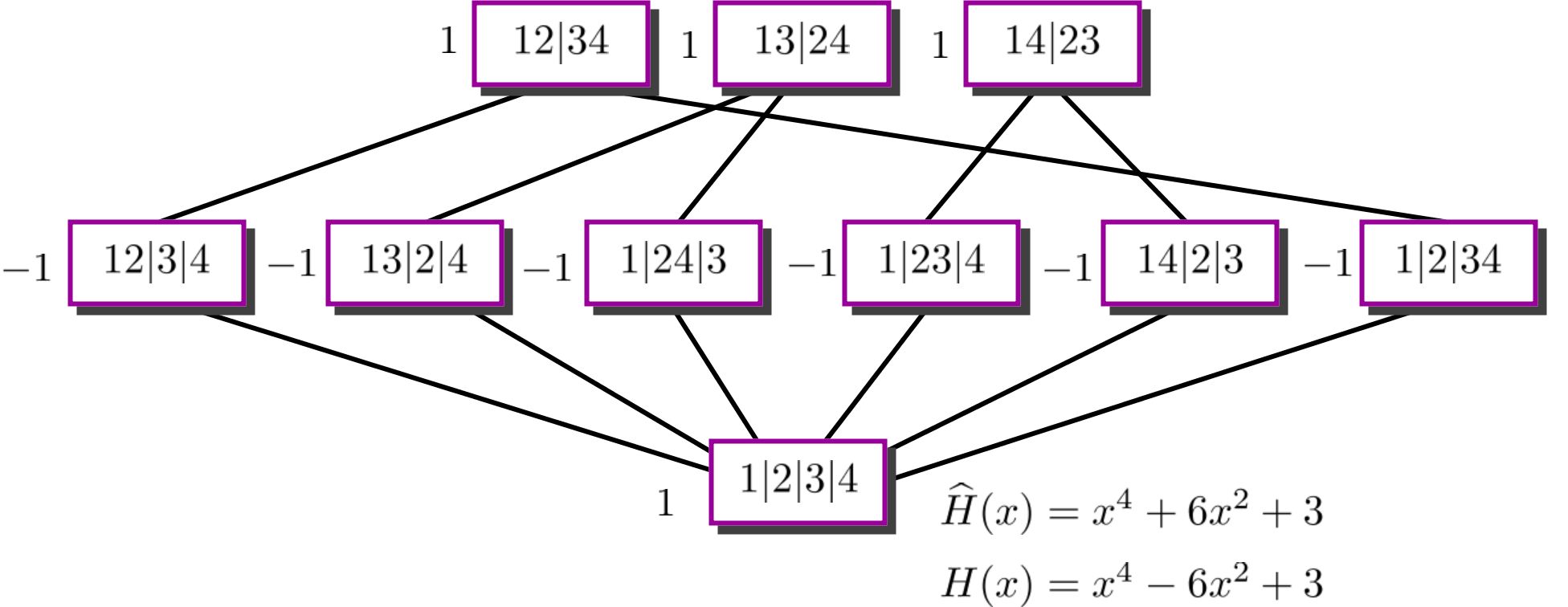}
\end{center}\caption{The poset $P_{E(E_2)}[3]$, and the Hermite polynomial $H_4(x)$.}\label{fig.hermite}		
\end{figure}
\begin{ex}\label{ex.Bell-appel}\normalfont{\em Bell-Appel polynomials.} The free commutative monoid generated by $E_+$, is equal to the species of partitions $\Pi=E(E_+)$. The Bell-Appel polynomials conjugate to $\Pi(x)=e^{e^x-1}$ are
	\begin{equation*}
	\mathfrak{B}_n(x)=\sum_{k=0}^n \binom{n}{k}B_k x^{n-k}.
	\end{equation*}
The M\"obius function of $P_{\Pi}[n]$ is equal to $\mu(\hat{0},\pi)=(-1)^{|\pi|}$. 
Then, the umbral inverse $\widehat{	\mathfrak{B}}_n(x)$ is equal to
\begin{equation*}
\widehat{\mathfrak{B}}_n(x)=x^n+\sum_{k=1}^n\binom{n}{k}(\sum_{j=1}^k (-1)^jS(k,j))x^{n-k}.
\end{equation*}
\end{ex}
\begin{ex}\label{ex.graph.appel}\normalfont Consider the species of graphs $\mathscr{G}$. Since we have the identity  $\mathscr{G}=E(\mathscr{G}_c)$, $\mathscr{G}_c$ being the species of connected graphs, the monoidal structure defined by $\nu_1$ in Ex. \ref{ex.graphs} is that of the free commutative monoid generated by $\mathscr{G}_c$. The corresponding Appel polynomials (conjugate to $\mathscr G(x)$) are 
	\begin{equation*}
	\widehat{\mathfrak{g}}_n(x)=\sum_{k=0}^n \binom{n}{k}2^{\binom{k}{2}}x^{n-k}.
	\end{equation*}
	The M\"obius function of $P_{\mathscr{G}}[n]$ is given by 
	$\mu(\widehat{0},G)=
	(-1)^{k(G)}$, where $k(G)$ is the number of connected components of $G$ (The empty graph is assumed to have zero connected components).  
	Their umbral inverses are the polynomials
	\begin{equation*}
	\frak{g}_n(x)=\sum_{k=1}^n\binom{n}{k}( \sum_{j=1}^{k}(-1)^j|\gamma_j(\mathscr{G}_c)[k]|)x^{n-k},
	\end{equation*}
	\noindent $\gamma_j(\mathscr{G}_c)$ being the species of graphs having exactly $j$ connected components.
	 
	\end{ex}
\begin{ex}\label{ex.monoidlist}\normalfont
	The species of lists $\LL$ (totally ordered sets) is a $c$-monoid with product $\nu:\LL.\LL\rightarrow\LL$, the concatenation of lists. The poset $P_{\LL}[n]$ has as maximal elements  the lists on $[n]$. We have that $l_1\leq l_2$ if $l_1$ is an initial segment of $l_2$. The M\"obius function is as follows,
	\begin{equation*}
	\mu(\widehat{0},l)=\begin{cases}1&\mbox{if $l$ is the empty list}\\-1& \mbox{if $l$ is a singleton}\\ 0&\mbox{otherwise.}\end{cases}
	\end{equation*}
	The conjugate polynomials of $\LL(x)=\frac{1}{1-x}$ are $$\widehat{a}_n(x)=\sum_{l\in P_{\LL}[n]}x^{n-|l|}=\sum_{k=0}^n\binom{n}{k}k! x^{n-k}=\sum_{k=0}^n (n)_k x^{n-k} =\sum_{k=0}^{\infty}D^kx^n.$$ Their umbral inverses are
	$$a_n(x)=\sum_{l\in P_{\LL}[n]}\mu(\hat{0},l)x^{n-|l|}=x^n-nx^{n-1}.$$
 \end{ex}
		
\subsection{Examples of $c$-operads and binomial families} 

\begin{ex}\label{ex.operadlists}The operad of lists and binomial Laguerre polynomials.\normalfont
	The species of non-empty lists $\LL_+$ is an operad (the associative operad) with $\eta$ the concatenation of lists following the order given by an external list, $l_{\pi}=B_1 B_2\dots B_k$
	$$\eta(\{l_B\}_{B\in\pi},l_{\pi})=l_{B_1}l_{B_2}\dots l_{B_k}.$$ 
	The elements of $P_{\LL_+}[n]$ are linear partitions (partition with a total order on each block). The coefficients counting  such linear partitions having $k$ blocks are the Lah numbers $$\binom{n-1}{k-1}\frac{n!}{k!}.$$ Hence, the polynomials obtained by summation on $P_{\LL_+}[n]$ are the unsigned Laguerre polynomials (of binomial type)
	\begin{equation}
	p_n(x)=\sum_{\pi\in P_{\LL_+}[n]}x^{|\pi|}=\sum_{k=1}^n\binom{n-1}{k-1}\frac{n!}{k!}x^k=L_n(-x).
	\end{equation}
	Since $\mu(\widehat{0},\pi)=(-1)^{n-|\pi|}$, by M\"obius inversion we get that
	\begin{equation}
	\widehat{p}_n(x)=\sum_{k=1}^n(-1)^{n-k}\binom{n-1}{k-1}\frac{n!}{k!}x^k=(-1)^nL_n(x).
	\end{equation} 
\end{ex}

\begin{ex}\normalfont\label{ex.Touchard}
	{\em Touchard polynomials.} The operad $E_+$ gives rise to the poset $E_+(E_+)[n]=\Pi_+[n]$ of non-empty partitions ordered by refinement.  The Touchard polynomials $T_n(x)$ conjugate to $E_+(x)=e^x-1$ are
	$$T_n(x)=\sum_{k=1}^n |\gamma_k(E_+)[n]|x^k=\sum_{k=1}^nS(n,k)x^k.$$ \noindent Where $S(n,k)$ are the Stirling numbers of the second kind. By M\"obius inversion we obtain their umbral inverses
	$$(x)_n=\sum_{k=1}^n |\gamma_k(E_+)[n]|_{\mu}x^k=\sum_{k=1}^ns(n,k)x^k,$$ where $s(n,k)$ is the Stirling number of the first kind and $(x)_n=x(x-1)(x-2)\dots (x-n+1)$ the falling factorial.
\end{ex}
\begin{ex}\label{ex.sinhop}\normalfont
	The species of sets having odd cardinal, $E^{\mathrm{odd}}$ inherit the operad structure from $E_+$. Hence, it is a $c$-operad. The poset $\Pi_{\mathrm{odd}}[n]=P_{E^{\mathrm{odd}}}[n]$ is formed by the partitions of $[n]$ where each block have odd length, ordered by refinement. Since $$E^{\mathrm{odd}}(x)=\frac{e^x-e^{-x}}{2}=\sinh(x),$$ the substitutional inverse of $E^{\mathrm{odd}}(x)$ is the arcsin series,
	$$\arcsin(x)=\ln(x+\sqrt{1+x^2})=\sum_{n=0}^{\infty}(-1)^n(2n-1)!!^2\frac{x^{2n+1}}{(2n+1)!}.$$
	The associated polynomials codify the M\"obius function of the poset $\Pi_{\mathrm{odd}}[n]$
	\begin{equation*}
	\mathfrak{st}_n(x)=x\prod_{k=1}^{n-1}(x+n-2k)=\sum_{\pi\in\Pi_{\mathrm{odd}}[n]}\mu(\hat{0},\pi)x^{|\pi|}.
	\end{equation*}
	It may be easily checked that $\frac{e^D-\,e^{-D}}{2}\mathfrak{st}_n(x)=\frac{\mathfrak{st}_n(x+1)-\mathfrak{st}_n(x-1)}{2}=n\mathfrak{st}_{n-1}(x).$ They are related to the Steffensen polynomials, \cite{Roman-Rota}, Ex. 6.1., by

		\begin{equation*}	\mathfrak{st}_n(x)=2^n\mathrm{st}_n(\frac{x}{2})
		\end{equation*}
		
\end{ex}
\begin{ex}\label{ex.cycles}\normalfont {\em The operad of cycles.}  Consider the rigid species of cyclic permutations $C$,
	\begin{equation}
	C[{v_1,v_2,\dots,v_n}]=\{f|f:V\rightarrow V \mbox{ a cyclic permutation }\}
	\end{equation}
	A cyclic permutation can be identified with a linear  order $l$ having $v_1$ as first element, $l_1=v_1$. 
	Its generating function is 
	 $C(x)=\ln(\frac{1}{1-x})$. It is a shuffle $c$-operad with product 
	 $\eta(\{l_B\}_{B\in\pi},l_{\pi})=l_{B_1}l_{B_{i_2}}\dots l_{B_{i_k}},$
	 the concatenation of the internal linear orders following the external order, $l_{\pi}=B_1 B_{i_2}\dots B_{i_k}$. Since $B_1$ is the first element of the totally ordered set $\pi=B_1<B_2<\dots<B_k$, the minimun element of $B_1$ is $v_1$, and the product gives again a cyclic permutation. 
	 
	 The elements of the poset $P_C[n]$ are permutations (assemblies of cyclic permutations), hence the conjugate sequence of  $C(x)=\ln(\frac{1}{1-x})$ is the increasing factorial, 
	 $$(x)^{\overline{n}}=x(x+1)\dots(x+n-1)=\sum_{k=1}^n|s(n,k)|x^k,$$
	 $s(n,k)$ being the Stirling numbers of the first kind. 
	 A cycle $c=(c_1c_2\dots c_j)$ of a permutation $\sigma$ is said to be monotone if $c_1<c_2<\dots<c_j$.
	 If $\sigma$ is a permutation with $k$ cycles, the M\"obius function was proved to be (see \cite{JoniRS}) 
	 \begin{equation*}
	 \mu(\hat{0},\sigma)=\begin{cases}
	 (-1)^{n-k}&\mbox{if all the cycles of $\sigma$ are monotone,}\\ 0&\mbox{ otherwise}.
	 \end{cases}
	 \end{equation*}
	 Hence, their umbral inverses are
	 \begin{equation*}
	 \sum_{k=1}^n (-1)^{n-k}S(n,k)x^k=(-1)^nT_n(-x)
	 \end{equation*}
	 \noindent $T_n(x)$ being Touchard polynomials.
\end{ex}

\begin{ex}\normalfont
	The Abel sequence $A_n(x;a)=x(x+a)^{n-1}$ associated a $xe^{-ax}$. For $a=1$, it is conjugate to the generating series  $\mathscr{A}(x)$ of rooted trees. The species of rooted trees has a $c$-operad structure (see \cite{Mend-Yang}). In \cite{Reiner} the poset $P_{\mathscr{A}}[n]$ was constructed, and its M\"obius function was computed in \cite{Sagan}.
\end{ex}
\begin{ex}
	The Bessel polynomials of Krall and Frink $y_n(x)$. \normalfont If we make $K_n(x)=x^ny_{n-1}(\frac{1}{x})$ it is the associate sequence of $x-\frac{x^2}{2}$. Hence, the conjugate to  $B(x)$, $B$ being the species of commutative parethesizations, or commutative binary trees, satisfying the implicit equation   
	\begin{equation*}B=X+E_2(B).\end{equation*}
	It is the free operad generated by $E_2$, a $c$-operad with the substitution of commutative parethesizations (or the grafting of commutative binary trees). Computing the inverse of $P(x)=x-\frac{x^2}{2}$ we obtain that 
	$B(x)=1-\sqrt{1-2x^2}$.
	The polynomials $K_n(x)$  have the following combinatorial interpretation,
	\begin{equation*}
	K_n(x)=\sum_{k=1}^n B_{n,k}x^k,
	\end{equation*}\noindent where $B_{n,k}$ is the number of forests having $k$ commutative binary trees with $n$ labeled leaves. The M\"obius function of such forests is
	\begin{equation*}
	\mu(\hat{0},a)=\begin{cases}0&\mbox{if $a$ has a tree with more than two leaves}\\
	(-1)^k&\mbox{$k$=number of binary trees with two leaves.}\end{cases}
	\end{equation*}
	
	Their umbral inverse is the family $\widehat{K}_n(x),$
	\begin{equation*}
	\widehat{K}_n(x)=\sum_{k=1}^{\lfloor \frac{n}{2}\rfloor}\binom{n}{2k}(2k-1)!! (-1)^k x^{n-k}.
	\end{equation*}
	  
\end{ex}
\begin{ex}\normalfont The generating function $\Gr_c(x)=\ln(1+\sum_{k=1}^{\infty}2^{\binom{n}{2}}\frac{x^n}{n!}),$ of the $c$-operad $\Gr_c$ of connected graphs (Ex. \ref{ex.graphs}) has as conjugate the binomial family
	\begin{equation*}
	\widehat{G}_n(x)=\sum_{k=1}^n |\gamma_k(\Gr_c)[n]|x^k.
	\end{equation*}
\noindent They are the generating function of graphs according to the number of their connected components. An explicit expression for their umbral inverses 
\begin{equation*}
G_n(x)=\sum_{k=1}^n |\gamma_k(\Gr_c)[n]|_{\mu}x^k,
\end{equation*} 
\noindent is not known.
\end{ex}
\subsubsection{The Dowling operad}\label{sec.dowoperad}
Let $G$ be a finite group of order $m$. Denote by $\Dowl$ the rigid species of $G$-colored ordered sets with an extra condition. The minimun element  of the set is colored with the identity $1$ of $G$. More explicitly, $\Dowl[\emptyset]=\emptyset$, and for a nonempty totally ordered set $V=\{v_1,v_2,\dots,v_n\}$, 
\begin{equation}
\Dowl[V]=\{f|f:V\rightarrow G,\; f(v_1)=1\}
\end{equation}
This kind of colorations will be called {\em unital}. 
It has as generating function 
\begin{equation}
\Dowl(x)=\sum_{n=0}^{\infty}m^{n-1}\frac{x^n}{n!}=\frac{1}{m}\sum_{n=0}^{\infty}\frac{(mx)^n}{n!}=\frac{e^{mx}-1}{m}.
\end{equation}
This species has a structure of $c$-operad, $\eta:\Dowl(\Dowl)\rightarrow \Dowl$, given as follows. The structures of $\Dowl(\Dowl)$ are pairs of the form $(\{f_B\}_{B\in\pi},g_{\pi})$ where each $f_B$ is a unital coloration on $B$, and $g_{\pi}$ is a unital coloration on $\pi$ (recall that $\pi$ is a totally ordered set, $B_1<B_2<\dots <B_k$, ordered according with their minimun element).
The product $h_V=\eta(\{f_B\}_{B\in\pi}, h_{\pi})$ is obtained by multiplying by the right the ``internal'' colors on each block $B\in \pi$ given by $f_{B}$,  times the ``external'' one given by $h(B)$. Let $b\in V$, and $B$ the unique block of $\pi$ where it belongs. Then define $h_V(b)$ by 
\begin{equation}
h_V(b)=f_{B}(b)\cdot h_{\pi}(B),
\end{equation}

\noindent where ``$\cdot$'' is the product of the group. A unital coloration can be represented as a monomial with exponents on $G$. The elements of $\Dowl(\Dowl)[V]$ are then identified with  factored monomials. This notation provides a better insight on the structure of the operad.
\begin{eqnarray}
(V,f)&\equiv&\prod_{v\in V}v^{f(v)}\\
(\{f_B\}_{B\in\pi}, h_{\pi})&\equiv& \prod_{B\in \pi}\left(\prod_{b\in B}b^{f_B(b)}\right)^{h_{\pi}(B)} \\
\eta:\prod_{B\in \pi}\left(\prod_{b\in B}b^{f_B(b)}\right)^{h_{\pi}(B)} &\mapsto&\prod_{B\in \pi}\prod_{b\in B}b^{f_B(b)\cdot h_{\pi}(B)}
\end{eqnarray}
For example, for the multiplicative group of non-zero integers module $5$, $G=\mathbb{Z}_5^*$, and $V=\{a,b,c,d,e,f,g,h\}$ we have:
\begin{equation}\label{factormonomial}\eta: (a^1 b^2 d^2)^1(c^1g^3f^4)^2(e^1h^3)^3\mapsto a^1b^2d^2c^2g^1f^3e^3h^4.\end{equation}
In  Dowling's original setting of lattices associated to a finite group, he made use of equivalence classes of colorations over partial partitions of a set. If we had followed his approach this would have led us to the definition of an equivalence relation between $G$-colorations, $f,h:V\rightarrow G$, $f\sim h$ if there exists a $g\in G$ such that $f=g\cdot h$. Observe that in each equivalence class of colorations there is only one which is unital. This is the reason why we define the Dowling operad by means of unital colorations. It is the natural way of avoiding complications with equivalence classes, by choosing one simple representative.

Since $G$ satisfies the left cancellation law, and $|E^{G}_+[1]|=1$, we have that $\Dowl$ is a $c$-operad:
$$\eta(\{f_B\}_{B\in\pi},h_{\pi})=\eta(\{f_B\}_{B\in\pi},h'_{\pi})\Rightarrow h_{\pi}=h'_{\pi},$$
we can define a posets $P_{\Dowl}[V]=(E(\Dowl)[V],\leq)$. The elements of the poset are assamblies of unital colorations, i.e., unital factored monomials. We say that $a_1\leq a_2$ if there exists a factored monomial $a_2'$ over the factoras of $a_1$ such  that $\bar{\eta}(a_1, a_2')=a_2$. For example for $G=\mathbb{Z}_5^{\ast}$, and naming $A=a^1b^2$, $B=c^1d^3$,  $C=e^1f^2g$, and $D=hc^2$ and consider the factored monoid 
\begin{equation*}
[A^1B^3][C^1D^2]
\end{equation*}
We have the product 
\begin{eqnarray*}
	\bar{\eta}((a^1b^2)(c^1d^3)(e^1f^2g)(h^1c^2),[A^1B^3][C^1D^2])&=&\eta((a^1b^2)^1(c^1d^3)^3)\eta((e^1f^2g)^1(h^1c^2)^2)\\&=&(a^1b^2c^3d^4)(e^1f^2gh^2c^4)
\end{eqnarray*}
Then, we have
$$(a^1b^2)(c^1d^3)(e^1f^2g)(h^1c^2)\leq (a^1b^2c^3d^4)(e^1f^2gh^2c^4)$$
The poset $P_{\Dowl}[n]$ has a unique minimal element $\wh{0}$, the assembly of trivial colorations over singletons, and $m^{n-1}$ maximal elements (the number of unital $G$ colorations). The exponential generating function of the M\"obius evaluation of $P_{\Dowl}[V]$, $$\mathrm{M}\ddot{\mathrm{o}}\mathrm{b}P_{\Dowl}[n]=\sum_{f\in \Dowl[n]}\mu(\hat{0},f),$$ 
is the substitutional inverse of the generating function $$\Dowl(x)=\frac{e^{mx}-1}{m},$$
\begin{equation}
\mathrm{M}\ddot{\mathrm{o}}\mathrm{b}P_{\Dowl}(x)=\sum_{n=1}^{\infty}\mathrm{M}\ddot{\mathrm{o}}\mathrm{b}P_{\Dowl}[n]\frac{x^n}{n!}=\ln(1+mx)^{\frac{1}{m}}=\sum_{n=1}^{\infty}(-m)^{n-1}(n-1)!\frac{x^n}{n!}.
\end{equation}
The binomial family conjugate to $\Dowl(x)=\frac{e^{mx}-1}{m}$ is the $[m]$-Touchard (see \cite{MeRam}),
$$T^{[m]}_n(x)=\sum_{a\in P_{\Dowl}[n]}x^{|a|}=\sum_{k=1}^nS^{[m]}(n,k)x^k.$$
Their umbral inverses being
$$ \widehat{T}^{[m]}_n(x)=\sum_{a\in P_{\Dowl}[n]}\mu(\widehat{0},a)x^{|a|}=\sum_{k=0}^ns^{[m]}(n,k)x^k=x(x-m)(x-2m)\dots(x-(n-1)m).$$

\section{Monops}\label{sec:monops} At this stage, having studied two particular cases, what is missing is a a general construction of families of posets in order to give a combinatorial interpretation to the umbral inversion for Sheffer families. Or equivalently, to the inverses of Riordan arrays. To this end we define
monoids in the Riordan category $\spe\rtimes\spe_+$. They will be called {\em \textcolor{blue}{monops}}, because they are an interesting mix between a monoidal structure in the first component of the pair, with an operad structure in the second one. 

\begin{defi}\normalfont A \textcolor{blue}{\em monop} is a monoid in the Riordan category $\spe \rtimes\spe_+$. More specifically, an admissible pair of species $(M,\OO)$ is called a monop if it is accompanied with a  product $(\rho,\eta)$, and identity morphisms $(\mathfrak{e},e)$, \begin{eqnarray}(\rho,\eta)&:&(M,\OO)\ast(M,\OO)\rightarrow (M,\OO) \mbox{ (product)}\\ (\mathfrak{e},e)&:&(1,X)\rightarrow (M,\OO)\mbox{ (identity) } 
	\end{eqnarray}
	satisfying the identity and associativity properties of a monoid in the context of the Riordan category  $\spe\rtimes \spe_+.$\end{defi}
We then have four natural transformations
\begin{equation*}
\rho:M\cdot M(\OO)\rightarrow\OO,\;\eta:\OO(\OO)\rightarrow \OO
\end{equation*}
\begin{equation*}
\mathfrak{e}:1\rightarrow M,\; e:X\rightarrow \OO.
\end{equation*}
That suggest, without looking at the commuting diagrams implicit in the definition of $(M,\OO)$, an operad structure on $\OO$,  a monoid structure on $M$, and some extra conditions.
We begin with two definitions in order to formulate those extra conditions.
\begin{defi}Right module over an operad.\\\normalfont
  Let $\OO$ be an operad and $M$ a species. We say the $M$ is a right module over $\OO$ if we have an action $\tau:M(\OO)\rightarrow M$ of $\OO$ over $M$ that is pseudo associative and where the assembly of identities of $\OO$ fixes every structure of $M$
\begin{eqnarray}
& &\tau(\bar{\eta}(a_1,a_2),m_{\pi})=\tau(a_1,\tau(	a_1,m_{\pi})),
\\
& &\tau(\{e_v\}_{v\in V}, m_V)=m_V.
\end{eqnarray}   	
\end{defi}
 For a detailed study of modules over operads and applications see \cite{Fressemodul}.
\begin{defi}Compatibility  condition.\\\normalfont Let $(M,\nu,\frak{e})$ be a monoid which is simultaneously a right module over $\OO$. We say that $\nu$ and $\tau$ are compatible if for every pair $$((a_{V_1},m_{\pi_1}), (a_{V_2},m_{\pi_2}))\in M(\OO)[V_1]\times M(\OO)[V_2]\subseteq M(\OO)\cdot M(\OO)[V],$$ we have
\begin{equation}
\nu(\tau(a_{V_1},m_{\pi_1}), \tau(a_{V_2},m_{\pi_2}))=\tau(a_{V_1}\sqcup a_{V_2},\nu(m_{\pi_1},m_{\pi_2})).
\end{equation}
\end{defi}

 \begin{theo}\label{teo.fundamentalmonops}Fundamental theorem of monops. \normalfont
 	Let $(M,\nu,\mathfrak{e})$ be a monoid,  and $(\OO,\eta,e)$ an operad, $M$ being a right $\OO$-module, $\tau:M(\OO)\rightarrow M$. If $\tau$ and $\nu$ are compatible then the pair $(M,\OO)$, $(M,\OO)=((M,\OO),(\rho,\eta), (\mathfrak{e},e))$, with $$\rho:=\nu\circ (M.\tau)$$ is a monop. Conversely, if $(M,\OO)=((M,\OO),(\rho,\eta), (\mathfrak{e},e))$ is a monop, then $(\OO,\eta,e)$ is an operad, $(M,\nu,\mathfrak{e})$, $\nu=\rho\circ (M.M(e))$  is a monoid with a structure of right $\OO$-module $\tau=\rho\circ(\mathfrak{e}\cdot M(\OO))$ and $\nu$ and $\tau$ are compatible.
 \end{theo}
We postpone the proof of the Fundamental Theorem to Section \ref{section.diagrams}. 

\subsection{Examples}
\begin{ex}\label{ex.commonop}\normalfont The pair $(E,E_+)$ is a monop. The Boolean monoid $E$ is a right module over the operad $E_+$. There is a unique homomorphism $\tau_V:E(E_+)[V]\rightarrow E[V]$, $\tau(\pi,\{\pi\})=\biguplus_{B\in\pi}B=V$. It is easy to check that the module and monoid structure are compatible.
\end{ex}
\begin{ex}\normalfont The pair 
	$(\LL,\LL_+)$ is a monop. The module structure $$\tau_V:\LL(\LL_+)[V]\rightarrow \LL[V]$$ is defined as follows. If $V=\emptyset$, $\tau_{\emptyset}$ is trivially defined.  Otherwise we define $\tau$ a concatenation of linear orders as for the operad $\LL_+$. The concatenation product $\nu:\LL.\LL\rightarrow \LL$ is clearly compatible with $\tau$. \end{ex}
\begin{ex}\label{ex.graphsmonop}\normalfont
	Let $\Gr_c$ be the species of connected graphs. It is a $c$-operad with respect to the restriction of the product $\eta$ defined in Ex. \ref{ex.graphs}. The species of graphs is a $c$-monoid with respect to the product $\nu_1$ of Ex. \ref{ex.graphs}. It is also a right $\Gr_c$-module by restricting appropriately the product $\eta$ of Ex. \ref{ex.graphs}, to obtain $\tau:\Gr(\Gr_c)\rightarrow\Gr$. It is easy to check that both structures are compatible. Hence the pair 
	$(\Gr,\Gr_c)$ is a $c$-monop.
	 As a motivating example of the general procedure we will develop in Section \ref{secposetmonop}, we are going to define a partial order over  $\Gr.\,E(\Gr_c)[V]$. An element of $\Gr.\,E(\Gr_c)[V]=\Gr.\Gr$ is a pair $(a_1,a_2)=(g_1,\{g_B\}_{B\in\pi})$ of graphs (an arbitrary graph and an assembly of connected graphs). The first element of the pair is called the monoidal section, and the assembly $a$ the operadic section. We represent the pair $(g_1,\{g_B\}_{B\in\pi})$ by placing a double bar between the monoidal zone $g_1$ and the operadic one, and simple bars between the elements of the operadic zone $\{g_B\}_{B\in\pi}$ (se Fig. \ref{fig.graphmonop}).  
	 We say that   $$(g_1,\{g_B\}_{B\in\pi})\leq(g_2,\{g_C\}_{C\in\sigma})$$
	 if the assembly $\{g_B\}_{B\in\pi}$ can be split in two subassemblies 
	 $$\{g_B\}_{B\in\pi}=\{g_B\}_{B\in\pi^{(1)}}+\{g_B\}_{B\in\pi^{(2)}},$$ such that
	 \begin{enumerate}
	 	\item $g_2=\nu_1(g_1,\tau(\{g_B\}_{B\in\pi^{(1)}},g_{\pi^{(1)}}))$, for some $g_{\pi^{(1)}}$ in $\Gr[\pi^{(1)}]$
	   \item $\{g_C\}_{C\in\sigma}=\eta(\{g_B\}_{B\in\pi^{(2)}},g_{\pi^{(2)}})$ for some $g_{\pi^{(2)}}\in\Gr[\pi^{(2)}]$. Equivalently, $\{g_B\}_{B\in\pi^{(2)}}$ is less than or equal to	$\{g_C\}_{C\in\sigma}$, in the partial order defined by the operad $(\Gr_c,\eta).$
    \end{enumerate}
In other words, a part of the assembly in the operadic zone of the pair is `abducted' to the monoidal zone, and then transformed, by means of $\tau$, in an element of the monoid. Finally it is multiplied, by means of $\nu_1$, with the element that initially was in the monoidal zone. The other part of the assembly in the operadic zone, remains in it and then substituted by a bigger assembly (in the partial order defined by the operad).
\end{ex}
\begin{figure}
	\begin{center}
		\includegraphics[width=130mm]{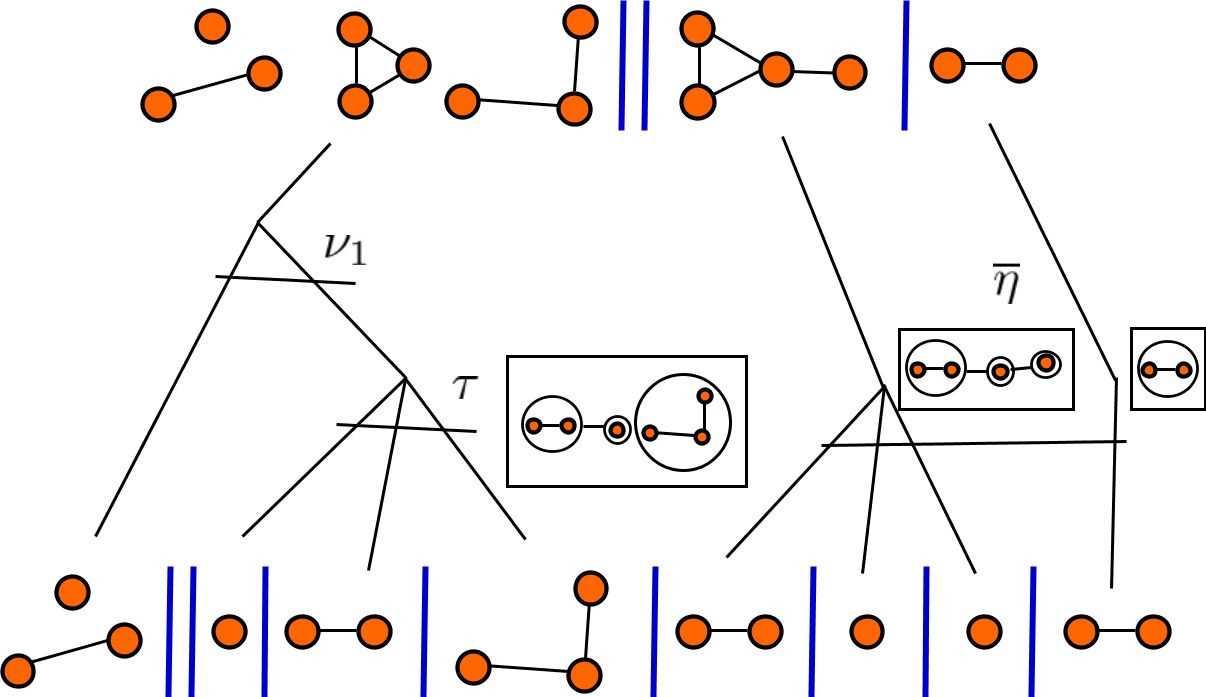}
	\end{center}\caption{Definition of the partial order on $\Gr.E(\Gr_c)[V]$ by means of the monop structure on $(\Gr,\Gr_c).$}\label{fig.graphmonop}
\end{figure}
 We will give a general construction of posets of this kind obtained from $c$-monops. Each of them gives us a Sheffer family and their umbral inverses via M\"obius inversion.  

\subsection{A family of monops: an operad and its derivative}\label{subsection.derivative}
For an operad $\OO$, the Riordan pair $(\OO',\OO)$ is a monop, 
$$(\rho, \eta):(\OO',\OO)\ast(\OO',\OO)=(\OO'.\OO'(\OO),\OO(\OO))\rightarrow (\OO',\OO),$$
where $\rho$ is the derivative of the morphism $\eta:\OO(\OO)\rightarrow \OO$, $\rho:=\eta'$.
In effect, by the chain rule we have that
\begin{equation}
\eta':\OO'.\OO'(\OO)\rightarrow \OO',
\end{equation}
and 
\begin{equation}
e':1\rightarrow \OO'.
\end{equation} 
The pair $(\OO',\OO)$ with the morphisms defined above is a monop (see Theorem \ref{teo.monopderivative} in Section \ref{section.diagrams}  ).  
\begin{ex}\normalfont
	The structure of monop $(E, E_+)$ in Ex. \ref{ex.commonop}  can be defined by the derivative procedure, since $E=E'_+.$
\end{ex}
\begin{ex}\normalfont The pair $(\LL',\LL_+)$, $\LL'=\LL^2$, is a monop.  \end{ex}

\section{Posets associated to $c$-monops}\label{secposetmonop}
\begin{defi}
	A monop $(M,\OO)$ is said to be  a $c$-monop if $\OO$ is a c-operad and $M$ is a c-monoid  and left cancellative as a right $\OO$-module.  
\end{defi}

For $c$-monop we will define a partially ordered set $P_{(M,\OO)}[V]$. 
Recall that the subjacent set of the partially ordered set $P_M[V]$ associated to a $c$-monoid $M$ is equal to $M.E[V]$, and that of $P_{\OO}[V]$, associated to a operad $\OO$ is $E_+(\OO)[V]$. By analogy we take the Riordan product with the monop $(E, E_+)$
\begin{equation}
(M,\OO)\ast (E,E_+)=(M.E(\OO),E_+(\OO))
\end{equation}
We already saw that $E_+(\OO)[V]$ is the set subjacent to the poset $P_{\OO}[V]$. The interesting posets associated to a monop are obtained by appropriately defining  an order over $M.E(\OO)[V]$.
Recall that the elements of $(M.E(\OO))[V]$ are pairs of the form $(m,a)$, where $(m,a)\in M[V_1]\times E(\OO)[V_2]$, for some decomposition of $V$ as a disjoint union  $V=V_1+V_2$.
Before defining it we require the following definition of product.
\begin{defi}\normalfont Let $(m_1,a_1)$ be an element of $(M.E(\OO))[V]$. Let $\pi$ be the partition subjacent to the assembly $a$. Let $(m_2',a_2')$ be an element of $M[\pi_1]\times E(\OO)[\pi_2]\subseteq M.E(\OO)[\pi]$, $\pi=\pi_1+\pi_2$ a splitting of $\pi$. Observe that either $\pi_1$ or $\pi_2$ may be empty. We define the product
	\begin{eqnarray}
	\nonumber\hro((m_1,a_1),(m_2',a_2'))&:=&(\nu(m_1,\tau(a_1^{(1)},m_2')), \etb(a_1^{(2)},a_2'))\\&=&(\rho(m_1,(a_1^{(1)},m_2')),\etb(a_1^{(2)},a_2')),
	\end{eqnarray}
	 where $a_1^{(i)}$ is the subassembly of $a_1$ having $\pi_i$ as subjacent partition, $i=1,2$. 
	\end{defi}
Observe that from the identity axioms for operads, monoids, and right $\OO$- modules we have that
\begin{equation}\label{eq.unibarrho}
\bar{\rho}((m,a),(\mathfrak{e},\{e_B\}_{B\in \pi}))=(a,m)=\bar{\rho}((\mathfrak{e},\{e_v\}_{v\in V}),(m,a)),
\end{equation}
$\pi$ being the partition subjacent to $a$.

\begin{theo}\label{teo:asocia}\normalfont The product $\hro$ is associative, left cancellative, and the identity does not have proper divisors. Let   $(m_1,a_1)$, $(m_2,a_2)$ and $(m_3,a_3)$ be a triplet of nested elements of $M.E(\OO)$, 
	\begin{enumerate}
		\item $(m_1,a_1)\in M[V_1]\times E(\OO)[V_2]\subseteq M.E(\OO)[V]$, $V=V_1+V_2$.
		\item $(m_2,a_2)\in M[\pi_1]\times E(\OO)[\pi_2],$ $\pi=\pi_1+\pi_2$ a splitting of $\pi$, the partition subjacent to the assembly $a_1$.
		\item $(m_3,a_3)\in M[\varsigma_1]\times E(\OO)[\varsigma_2]$, $\varsigma=\varsigma_1+\varsigma_2$ a splitting of $\varsigma$, the partition subjacent to the assembly $a_2$.
	\end{enumerate}
	We have 
	\begin{enumerate}
	\item Associativity

	\begin{equation}\label{eq.associativity}
	\hro(\hro((m_1,a_1), (m_2,a_2)),(m_3,a_3))=\hro((m_1,a_1), \hro((m_2,a_2),(m_3,a_3))).\end{equation}
	\item Left cancellation law 
	\begin{equation}
	\hro((m_1,a_1), (m_2,a_2))=\hro((m_1,a_1), (m'_2,a'_2))\Rightarrow (m_2,a_2)=(m'_2,a'_2).
	\end{equation}
		\item The identity does not have proper divisors
	\begin{equation}\label{eq.nondiv}
	\hro((m_1,a_1),(m_2,a_2))=(\mathfrak{e},\{e_v\}_{v\in V})\Rightarrow m_1=m_2=\mathfrak{e},\; a_1=a_2=\{e_v\}_{v\in V}.
	\end{equation}
	\end{enumerate}
\end{theo}	
	
\begin{proof}
	Le us prove associativity. We first introduce some notation. Let $a=\{\omega_B|B\in\sigma\}$ be an assembly with subjacent partition $\sigma$, and let $\sigma_1$ be a subset of $\sigma$. We denote by $(a)_{\sigma_1}$ the subset of $a$,
	$$(a)_{\sigma_1}=\{\omega_B|B\in\sigma_1\}.$$ For another partition $\varphi$, $\varphi\geq\sigma$, and $\varphi_1\subseteq\varphi$, $a_{\subseteq\varphi_1}$ is defined to be the subset of $a$,
	$$a_{\subseteq\varphi_1}=\{\omega_B|B\subseteq C,\; C\in \varphi_1\}.$$
	Computing $\hro((m1,a_1),(m_2,a_2))$ we get
	$$\hro((m_1,a_1),(m_2,a_2))=(\nu(m_1,\tau(a_1^{(1)},m_2)),\etb(a_1^{(2)},a_2)),$$
	where $a_1^{(i)}=(a_i)_{\pi_i}$, $i=1,2$. The assembly $\etb(a_1^{(2)},a_2)$ decomposes as a disjoint union
	$$\etb(a_1^{(2)},a_2)=\etb(a_1^{(2,1)},a_2^{(1)})\sqcup\etb(a_1^{(2,2)},a_2^{(2)}),$$
	where $a_1^{(2,i)}=(a_1^{(2)})_{\subseteq\varsigma_i}$ and $a_2^{(i)}=(a_2)_{\varsigma_i}$, $i=1,2$.
	Hence, 
	\begin{alignat}{2}\label{eq.asoist}
\nonumber	\hro(\hro((m_1,a_1)&, (m_2,a_2)),(m_3,a_3))=\\
	&(\nu\,(\nu\,(m_1,\tau(a_1^{(1)},m_2)),\tau(\etb(a_1^{(2,1)},a^{(1)}_2),m_3)),\etb(\etb(a_1^{(2,2)},a_2^{(2)}),a_3))
	\end{alignat}
	Since $M$ is a right $\OO$-module, we have $\tau(\etb(a_1^{(2,1)},a^{(1)}_2),m_3)=\tau(a_1^{(2,1)},\tau(a^{(1)}_2,m_3)).$  By associativity of $\nu$ and $\etb$, we get from Eq. (\ref{eq.asoist}) 
	\begin{alignat}{2}\nonumber\hro(\hro((m_1,a_1)&, (m_2,a_2)),(m_3,a_3))=\\
	&(\nu\,(m_1,\nu\,(\tau(a_1^{(1)},m_2)),\tau(a_1^{(2,1)},\tau(a^{(1)}_2,m_3)),\etb(a_1^{(2,2)},\etb(a_2^{(2)},a_3)))\end{alignat}
	\noindent From the compatibility between $\eta$ and $\tau$,
	$$\nu\,(\tau(a_1^{(1)},m_2),\tau(a_1^{(2,1)},\tau(a^{(1)}_2,m_3)))=\tau\,(\nu\,(a_1^{(1)}\sqcup a_1^{(2,1)},\nu\,(m_2,\tau(a^{(1)}_2,m_3)))).$$
	Hence
	\begin{alignat}{2}\nonumber\hro(\hro((m_1,a_1)&, (m_2,a_2)),(m_3,a_3))=\\
	&(\nu\,(m_1,\tau\,(\nu\,(a_1^{(1)}\sqcup a_1^{(2,1)},\nu\,(m_2,\tau(a^{(1)}_2,m_3)),\etb(a_1^{(2,2)},\etb(a_2^{(2)},a_3)))\end{alignat}
	
	The right hand side of Eq. (\ref{eq.associativity}) is equal to
	\begin{alignat}{2}\nonumber\hro((m_1,a_1),&\hro( (m_2,a_2),(m_3,a_3)))=\\
	&(\nu\,(m_1,\tau\,(\nu\,((a_1)_{\sigma_1},\nu\,(m_2,\tau(a^{(1)}_2,m_3)),\etb((a_1)_{\sigma_2},\etb(a_2^{(2)},a_3))).\end{alignat}
	 
	\noindent Since the partition $\sigma_1$ is the set of labels of $\nu\,(m_2,\tau(a^{(1)}_2,m_3))$, which is equal to  $$\pi_1\sqcup\{B|B\in\pi,\, B\subseteq C\in\varsigma_1\},$$ we have $(a_1)_{\sigma_1}=(a_1)_{\pi_1}\sqcup (a_1)_{\subseteq \varsigma_1}=a_1^{(1)}\sqcup a_1^{(2,1)}$. In the same way we get that $\sigma_2=\varsigma_2$, and that $(a_1)_{\sigma_2}=a_1^{(2,2)}.$

	 The left cancellation law follows easily from the left cancellation law for $\etb$, $\nu$ and $\tau$. The non existence of proper divisors of the identity is also easy and left to the reader. 
\end{proof}

 The partial order is defined as follows.
\begin{defi}\label{def:partialorder}\normalfont Let $(m_1,a_1), (m_2,a_2)$ be two elements in $M.E(\OO)[V]$. Let $\pi$ be the partition subjacent to $a_1$. We say that $(m_1,a_1)\leq (m_2,a_2)$ if there exists another pair $(m_2',a_2')\in M[\pi_1]\times E(\OO)[\pi_2]$,  $\pi=\pi_1+\pi_2$, such that:
	\begin{equation}
	\hro((m_1,a_1),(m_2',a_2'))=(m_2,a_2).
	\end{equation} 
\noindent Equivalently
\begin{eqnarray}\nonumber m_2&=&\nu(m_1, \tau(a_1^{(1)},m_2'))\\ a_2&=&\etb(a_1^{(2)},a'_2).
 \end{eqnarray} 
\end{defi}
\begin{prop}\normalfont The relation $(m_1,a_1)\leq (m_2,a_2)$ in Definition \ref{def:partialorder} is a partial order. 
\end{prop}	
	\begin{proof}
		By Eq. (\ref{eq.unibarrho}) reflexivity follows. Transitivity follows from associativity of $\hro$ in Theorem \ref{teo:asocia}. Antisymmetry follows from the left canellation law and the non-divisibility of the identity (Eq. (\ref{eq.nondiv}))
		\end{proof}
	
 \begin{prop}\label{propertiesposetmonop}\normalfont The family of posets $\{P_{(M,\OO)}[V]|V\mbox{ a finite
 			set}\}$ satisfies the following properties:
 		\begin{enumerate}
 			\item $P_{(M,\OO)}[V]$ has a $\hat{0}$ equal the pair $(\mathfrak{e},a_0)$, $\mathfrak{e}$ the unique element of
 			$M[\emptyset]$, and $a_0=\{e_v\}_{v\in V}$ the unique assembly of $E(\OO)[V]$ formed by  singleton structures of $\OO$. Its elements of the form $(m,\emptyset)$, $m\in M[V]$ are maximal. 
 			\item If $f:V\rightarrow U$ is a bijection, $P_{(M,\OO)}[f]:P_{(M,\OO)}[V]\rightarrow P_{(M,\OO)}[U]$ is an order
 			isomorphism. 
 			
 			\item For  $(m_1,a_1)$ an element of $P_{(M,\OO)}[V]$,
 			the order coideal
 			$$\mathscr{C}_{(m_1,a_1)}=\{(m_2,a_2)\in P_{(M,\OO)}[V]|(m_2,a_2)\geq (m_1,a_1)\},$$ is isomorphic to
 			$P_{(M,\OO)}[\pi]$, $\pi$ being the partition subjacent to $a_1$. 
 			\item Every interval $[(m_1,a_1),(m_2,a_2)]$  of $P_{(M,\OO)}[V]$  is isomorphic to the
 			interval $[\hat{0},(m_2',a_2')]$ of $P_{(M,\OO)}[\pi]$, $(m_2',a_2')$ being the
 			unique element of $M.E(\OO)[\pi]$ such that $$\hro((m_1,a_1),(m_2',a_2'))=(m_2,a_2).$$
 			\item The interval $[\wh{0}, (m,a)]$ of $P_{(M,\OO)}[V]$, $(m,a)=(m,\{\omega_B\}_{B\in\pi})\in M[V_1]\times E(\OO)[V_2]$ is isomorphic to the product
 			\begin{equation*}
 			[\hat{0},(m,\emptyset)]_{V_1}\times \prod_{B\in \pi}[\hat{0},\{\omega_B\}]_{B}
 			\end{equation*}
 			\end{enumerate}
 		\end{prop}
 	\begin{proof}
 		Property 1 follows directly from Eq. (\ref{eq.unibarrho}). Property 2 from the equivariance of $\hro$. To prove Property 3, choose an arbitrary element $(m'_2,a'_2)$ in $P_{(M,\OO)}[\pi]$ and define $\phi((m'_2,a'_2)):=\hro((m_1,a_1),(m'_2,a'_2)).$ By the definition of the partial order, associativity,  and the left cancellation  law $\phi$ is an isomorphism. Property 4 is obtained in the same way by restricting $\phi$ to the interval $[\hat{0},(m_2',a_2')]$. To prove Property 5, first observe that the product $\prod_{B\in \pi}[\hat{0},\{\omega_B\}]_{B}$ is isomorphic to the interval $[\widehat{0},a]$, $a=\{\omega_B\}_{B\in\pi}$. Hence, we have to prove that the interval $[\widehat{0},(m,a)]$ is isomorphic to the product $[\widehat{0},(m,\emptyset)]_{V_1}\times[\widehat{0},a]_{V_2}$. For an arbitrary element, $((m_1,a_1),a_2)\in [\widehat{0},(m,\emptyset)]_{V_1}\times[\widehat{0},a]_{V_2}$,   $(m_1,a_1)\leq (m,\emptyset)$, and $a_2\leq a$. It means that $m=\nu(m_1,\tau(a_1,m_2))$, and that $)\etb(a_2,a')=a$
 		for some $m_2$ and some $a'$. Define $\psi((m_1,a_1),a_2)):= (m_1, a_1\sqcup a_2)\in [\widehat{0},(m,a)]$. It is easy to prove that $\psi$ is an isomorphism.
 	\end{proof}
 	Let $A$ be a subset of a poset $P_{(M,\OO)}[n]$. We define the M\"obius cardinal of $A$ as the sum 
 	$$|A|_{\mu}=\sum_{(m,a)\in A}\mu(\widehat{0},(m,a)).$$
 	\begin{theo}\label{teo.inversematrix}
 			 	Let $(M,\OO)$ be a $c$-monop. Then, the Riordan matrices $C_{n,k}=|M\cdot \gamma_k(\OO)[n]|$ and $\widehat{C}_{n,k}=|M\cdot \gamma_k(\OO)[n]|_{\mu}$, are one inverse of the other. Equivalently,  they are associated respectively to the Riordan pairs $(M(x),\OO(x))$ and $(M(x),\OO(x))^{-1}$.
 			 \end{theo}
 		 \begin{proof}
 		 	Le us consider the poset $P_{(M,\OO)}[n]$. By properties of the M\"obius function we have
 		 	\begin{equation*}
 		 	\sum_{\widehat{0}\leq (m_1,a_1)\leq (m_2,a_2)}\mu((m_1,a_1),(m_2,a_2))=\delta_{n,j}
 		 	\end{equation*}
 		 	\noindent where $(m_2,a_2)$ is any element of $P_{(M,\OO)}[n]$ such that $|a_2|=j\leq n$. Adding over all such elements in $P_{(M,\OO)}[n]$, we get
 		 	\begin{alignat}{4}\nonumber\delta_{n,j}=\sum_{(m_2,a_2)\in M\cdot\gamma_j(\OO)[n]}\;&\sum_{\widehat{0}\leq (m_1,a_1)\leq (m_2,a_2)}\mu((m_1,a_1),(m_2,a_2))=\nonumber\\\nonumber\sum_{(m_2,a_2)\in M\cdot\gamma_j(\OO)[n]}&\;\sum_{j\leq k\leq n}\;\sum_{\widehat{0}\leq (m_1,a_1)\leq (m_2,a_2),\; |a_1|=k}\mu((m_1,a_1),(m_2,a_2))=\\\nonumber\sum_{j\leq k\leq n}\;\sum_{(m_1,a_1)\in M\cdot\gamma_k(\OO)[n]} &\;\sum_{(m_2,a_2)\in M\cdot\gamma_j(\OO)[n]}\;\sum_{(m_2,a_2)\geq(m_1,a_1) }\mu((m_1,a_1),(m_2,a_2))=\\\nonumber\sum_{j\leq k\leq n}\;\sum_{(m_1,a_1)\in M\cdot\gamma_k(\OO)[n]} &\;\sum_{(m'_2,a'_2)\in M\cdot\gamma_j(\OO)[k]}\mu(\widehat{0},(m'_2,a'_2))=\sum_{j\leq k\leq n}\;| M\cdot\gamma_k(\OO)[n]| & |M\cdot\gamma_j(\OO)[k]|_{\mu}.
 		 	\end{alignat}
 		 	The last equation follows from Proposition \ref{propertiesposetmonop} (properties  3 and 4), and the fact that if $\hro((m_1,a_1),(m'_2,a'_2))=(m_2,a_2),$ then $|a'_2|=|a_2|$.
 		 \end{proof}
 	\subsection{Examples}

 	\begin{ex}\normalfont{\em Actuarial polynomials.}
 		 Actuarial polynomials $a^{[\beta]}(x)$ are associated to $$((1-x)^{-\beta}, \ln(1-x)).$$
 		 
 	For $\beta=r$, a positive integer, we get that the Sheffer conjugate to $(E^r(x), E_+(x))=(e^{rx},e^x-1)$, are associated to $((1-x)^{-r},\ln(1+x)).$ Hence, since the Touchard polynomials $T_n(x)$ are associated to $\ln(1+x)$, from Eq. (\ref{eq:shefferbinomial}) the actuarial polynomials evaluated in $-x$ is equal to 
 	$$(1+D)^{r}T_n(x)=\sum_{k=1}^r\binom{r}{k}T^{(k)}(x).$$
 	The pair $(E^r,E_+)$ is a $c$-monop, $E^r$ being the ballot monoid in Ex. \ref{ex.Boolean}, and $E_+$ the commutative operad of \ref{ex.Touchard}. The action of $\tau:E^r(E_+)\rightarrow E_+$ given by $$\tau(\pi, (\pi_1,\pi_2,\dots,\pi_r))=(\cup_{B\in\pi_1}B,\cup_{B\in\pi_2}B,\dots, \cup_{B\in\pi_r}B),$$ 
 	$(\pi_1,\pi_2,\dots,\pi_r)$ being an $r$-composition of $\pi$.
 	
 	The elements of the partially ordered set $P_{(E^r,E_+)}[V]$ are pairs $(\mathbf{W}, \pi),$ where $\mathbf
 	{W}=(W_1,W_2,\dots,W_r)$ is a $r$-composition of some subset $V_1$ of $V$, and $\pi$ is a partition of its complement in $V$. The partial order $P_{(E^r,E_+)}[V]$ is better described by the covering relation.  We say that $(\mathbf{W},\pi)\prec (\mathbf{W}',\pi')$ if either,
 	\begin{enumerate}
 		\item There exist a block $B$ of $\pi$ and some $1\leq i\leq r$, such that $$\mathbf{W}'=(W_1,W_2,\dots,W_i+B,\dots,W_r),$$ and $$\pi'=\pi-\{B\}.$$
 		\item The partition $\pi'$ covers $\pi$ in the refinement order, and $\mathbf{W}'=\mathbf{W}$. 
 	\end{enumerate}
 	See Fig.\ref{fig.actuarial} 
 \begin{figure}
 	\begin{center}
 		\includegraphics[width=160mm]{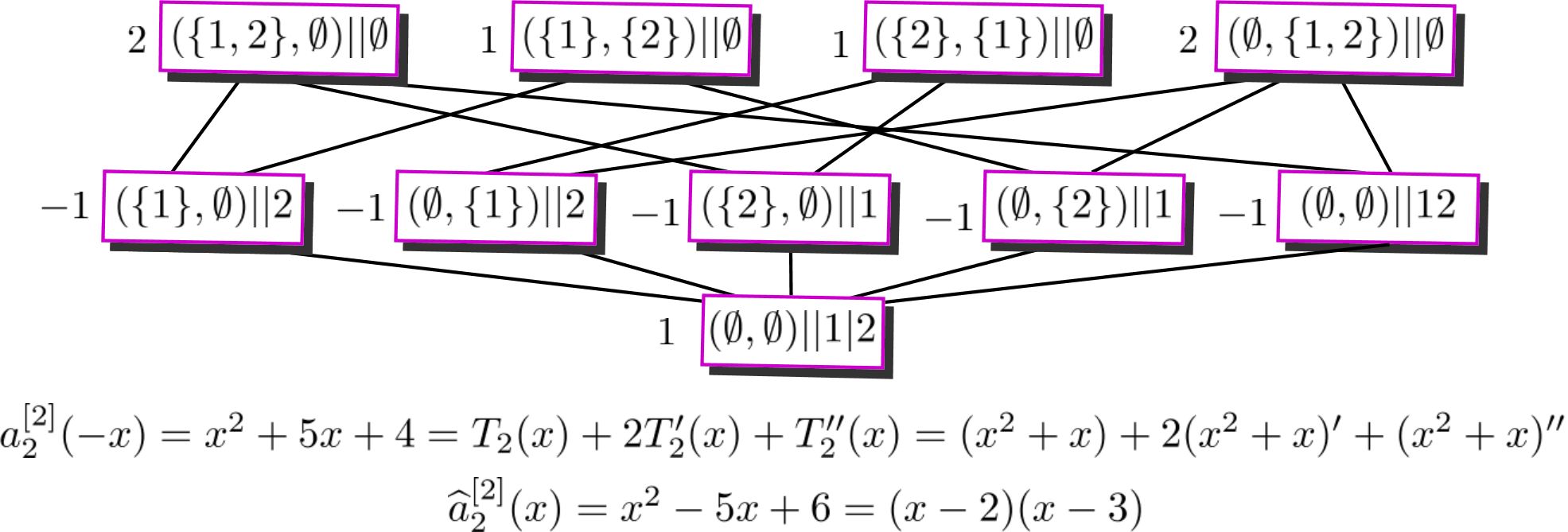}
 	\end{center}\caption{The poset $P_{(E^2,E_+)}[2]$, actuarial polynomial $a_2^{[2]}(x)$, and its umbral inverse $\widehat{a}_2^{[2]}(x).$}\label{fig.actuarial}
 \end{figure}
 	
 	 Their umbral inverses are the falling factorials $$(x-r)(x-r-1)\dots(x-r-n+1)=\sum_{\mathbf{W}||\pi\in P_{(E^2,E_+)}[n]}\mu(\widehat{0},\mathbf{W}||\pi) x^{|\pi|}.$$ 
 
 		\end{ex}
 \begin{ex}\normalfont {\em Laguerre polynomials $L^{[\alpha]}_n(x)$}. The Laguerre polynomials are Sheffer associated to  $(\frac{1}{(1-x)^{\alpha+1}},\frac{x}{x-1})$,$$L^{[\alpha]}(x)=\sum_{k=0}^{n}\binom{n+\alpha}{n-k}\frac{n!}{k!}(-x)^k.$$ For $r=\alpha+1$, a nonnegative integer,
 	$$L_n^{[r-1]}(x)=\sum_{k=0}^{n}\binom{n+r-1}{n-k}\frac{n!}{k!}(-x)^k=\sum_{k=0}^{n}\binom{n+r-1}{k+r-1}\frac{n!}{k!}(-x)^k.$$
 	Let us consider the pair $(\LL^r,\LL_+)$. $\LL^r$ is the $r$-power of the monoid of lists, Ex. \ref{ex.monoidlist}, and $\LL_+$ the operad of non-empty lists (the associative operad Ex. \ref{ex.operadlists}). It is a $c$-monop, $\LL^r$ a monoid with product $\nu$ the concatenation of r-uples of linear orders. It is also a compatible right $\LL_+$-module  with the action $$\tau(\{l_B\}_{B\in \pi},(l_{\pi_1},l_{\pi_2},\dots,l_{\pi_r}))=(l_1,l_2,\dots,l_r),$$  
 where $l_i$ is given by
 $$l_i=\begin{cases}
 \eta(\{l_B\}_{B\in\pi_i},l_{\pi_i})&\mbox{if $\pi_i\neq\emptyset$}\\\mathfrak{e}&\mbox{otherwise.}
 \end{cases}$$ 
 $\eta$ being the product of the operad $\LL_+$, and $\mathfrak{e}$ the empty order, $\mathfrak{e}\in\LL[\emptyset]$. The elements of the poset $P_{(\LL^r,\LL_+)[n]}$ are pairs of the form $(\mathbf{l},\pi)\in \LL^r.E(\LL_+)[n]$, where $\mathbf{l}=(l_1,l_2,\dots,l_r)$ is an $r$-uple of linear orders and $\pi$ is a linear partition. The numbers of such pairs satisfying $|\pi|=k$ is easily proved to be  $$\binom{n+r-1}{k+r-1}\frac{n!}{k!}.$$ Hence, the Sheffer polynomials obtained by summation over $P_{(\LL^r,\LL_+)}[n]$ are  
 \begin{equation*}
 \widehat{s}_n(x)=\sum_{\mathbf{l},\pi)\in P_{(\LL^r,\LL_+)[n]}}x^{|\pi|}=\sum_{k=1}^{n}\binom{n+r-1}{k+r-1}\frac{n!}{k!}x^k=L^{[r-1]}_n(-x).
 \end{equation*}
 The M\"obius function is equal to
 \begin{equation*}
 \mu(\hat{0},(\mathbf{l},\pi))=(-1)^{n-|\pi|},
 \end{equation*}
 \noindent By M\"obius inversion, their umbral inverse family is equal to
 \begin{equation*}
  s_n(x)=\sum_{(\mathbf{l},\pi)\in P_{(\LL^r,\LL_+)[n]}}(-1)^{n-|\pi|}x^{|\pi|}=(-1)^n\sum_{k=0}^n\binom{n+r-1}{k+r-1}\frac{n!}{k!}(-x)^k=(-1)^nL_n^{[r-1]}(x)
 \end{equation*}
 	\end{ex}
 	\begin{ex}\normalfont {\em Poisson-Charlier polynomials.}
 	\normalfont Consider the species of partitions $\Pi$. It is simultaneously the free commutative monoid generated by $E_+$, Ex. \ref{ex.Bell-appel}, and the free right $E_+$-module generated by $E$; $\Pi=E(E_+)$ (see Ex. \ref{ex.Bell-appel}). As a free right $E_+$ module, the  product is equal to $\tau=E(\eta)$,$$ \tau:E(E_+(E_+))\rightarrow E(E_+).$$ 
 			The monoid structure of $\Pi$ is easily seen to be compatible with this module structure. Hence $(\Pi,E_+)$ is a monop, more specifically, a $c$-monop. Its generating function and that of its inverse are the Riordan pairs
 			\begin{eqnarray*}
 				(\Pi(x),E_+(x))&=&(e^{e^x},e^x-1)\\
 				(\Pi(x),E_+(x))^{-1}&=&(e^{-x}, \ln(1+x)).
 			\end{eqnarray*}
 		
 		The poset $P_{(\Pi,E_+)}[n]$ has as subjacent set $\Pi.E(E_+)[n]=\Pi.\Pi[n]$. The elements of $\Pi.\Pi[V]$ are pairs of partitions $(\pi_1,\pi_2)$, $\pi_i\in\Pi[V_i]$, $i=1,2$, $V_1+V_2=[n]$.\\
 		 Let $(\pi_1, \pi_2)$ and $(\pi_3,\pi_4)$ be two pairs of partitions in $\Pi.\Pi[n]$.
 			We will say that $(\pi_1, \pi_2)\leq (\pi_3,\pi_4)$ if we can split $\pi_2$ in two partitions, $\pi_2=\pi_2^{(1)}+\pi_2^{(2)}$, such that
 			\begin{enumerate}
 				\item $\pi_3=\pi_1\uplus \pi_1'$, $\pi_1'$ being some partition on $V_1$ greater than or equal to $\pi_2^{(1)}$ in the refinement order.
 				\item The partition $\pi_4$ is greater than or equal to $\pi_2^{(2)}$ in the refinement order.
 			\end{enumerate}
 			\noindent See Fig. \ref{fig.charlier} for the poset  $(\Pi.\Pi)[\{1,2,3\}]$. A pair $(\pi_1,\pi_2)$ is represented by placing a double bar between the partitions. 
 			The partial order is better described by the covering relation. We will say that $(\pi_1,\pi_2)$ is covered by $(\pi_3,\pi_4)$ if either:
 			\begin{enumerate}
 				\item There exists a block $B$ in $\pi_2$ such that $\pi_3=\pi_2\uplus\{B\}$ and $\pi_4=\pi_2-\{B\}$.
 				\item The partition $\pi_4$ covers $\pi_2$ in the refinement order of partitions. That is,  $\pi_4$ is obtained by joining exactly two blocks of $\pi_2$.
 			\end{enumerate} 
 	
 		\begin{figure}
 			\begin{center}\includegraphics[width=160mm]{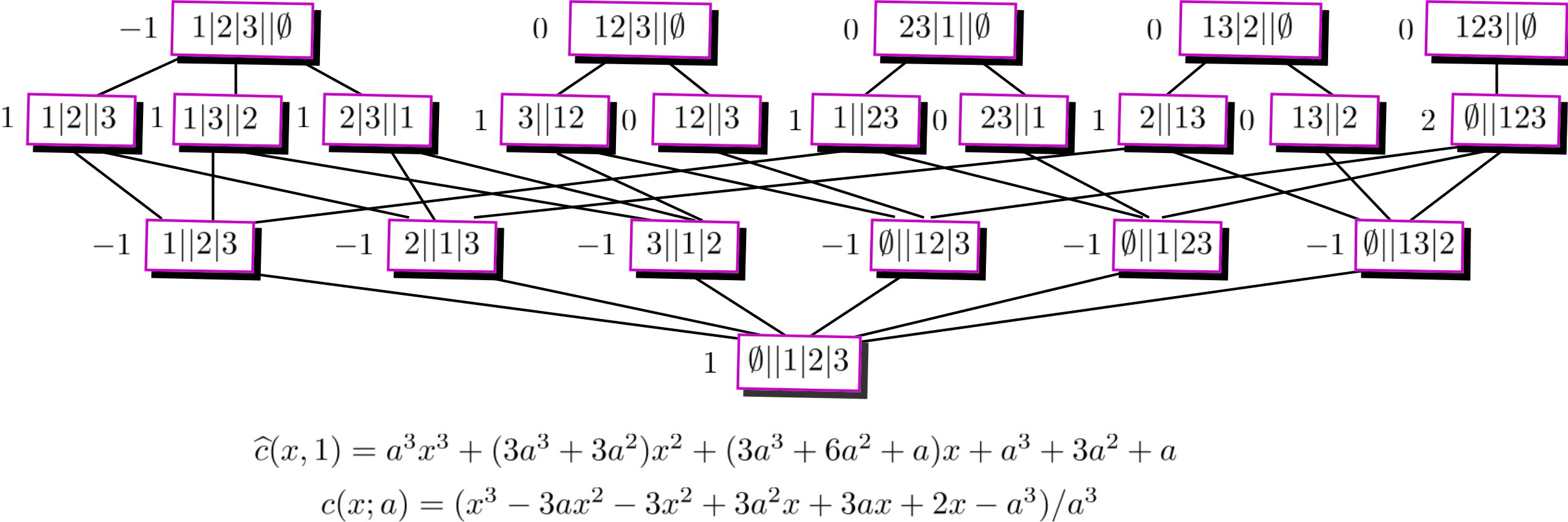}
 			\end{center}\caption{Poset associated to the monop $(\Pi,E_+)$, and Poisson-Charlier polynomials.}\label{fig.charlier}

 		\end{figure}
 		This family of posets gives us the combinatorics of the Poisson-Charlier polynomials and their umbral inverses.  By summation over $P_{(\Pi,E_+)}[n]$ we get the Shifted Touchard polynomials
 		$$T_n(x+1)\sum_{\pi_1||\pi_2\in P_{(\Pi,E_+)}[n]}x^{|\pi_2|}$$
 		By M\"obius inversion we get the Poisson-Charlier polynomials corresponding to the parameter $a=1$,
 		$$c_n(x;1)=\sum_{\pi_1||\pi_2\in \Pi.\pi[n]}\mu(\widehat{0},\pi_1||\pi_2)x^{\pi_2}.$$
 		
 		The general Poisson-Charlier polynomials $c_n(x;a)$ are the umbral inverses of
 		 the Sheffer family $T_n(ax+a)$,   
 		\begin{equation}
 		T_n(ax+a)=\sum_{\pi_1||\pi_2\in\Pi.\Pi[n]}a^{|\pi_1|+|\pi_2|}x^{|\pi_2|}.
 		\end{equation}
 		
 		The polynomials $c_n(x;a)$ have the following combinatorial interpretation in terms of the parameter $a$ and the M\"obius function of $P_{(\Pi,E_+)}[n]$ (see Fig. \ref{fig.charlier}).
 		\begin{equation}
 		c_n(x;a)=\frac{1}{a^n}\sum_{\pi_1||\pi_2\in \Pi\Pi[n]}\mu(\hat{0},\pi_1||\pi_2)a^{|\pi_1|}x^{|\pi_2|}.
 		\end{equation}
 		\end{ex}
 	\begin{ex}The hyperbolic monop. \normalfont
 		The pair $(E^{\mathrm{ev}},E^{\mathrm{odd}})$ is a $c$-monop. Its generating function is $(\cosh(x),\sinh(x))$ and inverse $(\frac{1}{\sqrt{1+x^2}},\ln(x+\sqrt{1+x^2}))$. In a forthcoming paper we will describe in detail the properties of the corresponding poset and associated Sheffer polynomials.
 	\end{ex}
 \begin{ex}\normalfont 
 	Consider the shuffle monop of lists and cyclic permutations $(\LL, C)$, $C'=\LL$. Its generating function  \begin{equation*}(\LL,C)(x)=(\frac{1}{1-x},\ln(\frac{1}{1-x}))\end{equation*}\noindent has as inverse \begin{equation}\label{eq:genlincy}(e^{-x}, 1-e^{-x})\end{equation} The elements of $P_{(\LL,C)}[n]$ are pairs of the form $(l,\sigma)\in \LL\cdot E(C)[n]$, $l$ a linear order and $\sigma$ a permutation. Since the  binomial family associated to $1-e^{-x}$ is the increasing factorial, Ex. \ref{ex.cycles}, the  Sheffer sequence associated to the Riordan pair in Eq. \ref{eq:genlincy}, by Eq. (\ref{eq:shefferbinomial}), is equal to $$e^{D}x(x+1)(x+2)\dots(x+n-1)=(x+1)(x+2)\dots (x+n)=\sum_{(l,\sigma)\in P_{(\LL,C)}[n]}x^{|\sigma|},$$
 	where $|\sigma|$ denotes the number of cycles in $\sigma$. Their umbral inverses codify the M\"obius function of $P_{(\LL,C)}[n]$,
 	\begin{eqnarray*}(1-D)(-1)^nT_n(-x)=(-1)^n(T_n(-x)+T'_n(-x))&=&(-1)^{n+1}\frac{T_{n+1}(-x)}{x}\\&=&\sum_{k=1}^nS(n+1,k)(-1)^{n+1-k}x^{k-1}\\&=&\sum_{k=0}^nS(n+1,k+1)(-1)^{n-k}x^{k}.\end{eqnarray*}
 	Hence:
 	$$|\LL.\gamma_k(C)[n]|_{\mu}=S(n+1,k+1)(-1)^{n-k}.$$
 \end{ex}
 	
 	\begin{ex}\normalfont The ballot monoid $E^r$ of Ex. \ref{ex.Boolean} together with the Dowling operad (subsection \ref{sec.dowoperad} )   form a monop $(E^r, \Dowl)$, that we call the {\em $r$-Dowling monop}. The monoid $E^r$ has also a structure of right $c$-$\Dowl$ module. 	\begin{eqnarray}
 		\nonumber \tau:E^r(\Dowl)&\rightarrow& E^r\\
 		\tau(\{f_B\}_{B\in\pi}, (\pi_1,\pi_2,\dots,\pi_r))&=&(\cup_{B\in\pi_1}B,\cup_{B\in\pi_2}B,\dots, \cup_{B\in\pi_r}B)
 		\end{eqnarray}
 		\noindent where $(\pi_1,\pi_2,\dots,\pi_r)$ is an $r$-composition of $\pi$, $(\pi_1,\pi_2,\dots,\pi_r)\in E^r[\pi].$
 		The reader may check that  $\nu$ and $\tau$ are compatible. For $r=1$ the pair $(E,\Dowl)$ will be called the {\em Dowling monop}. In the next subsection we will give details of the construction of the Dowling and the $r$-Dowling posets. Observe that  this example corresponds to the Riordan category in the context of $\mathscr{L}$-species with shuffle product and substitution, their underlying sets are totally ordered.
 	\end{ex}	
 	\subsection{The Dowling monop, Dowling lattices and the $r$-Dowling posets}\label{subsec:rdow} 
 	
 	The Dowling lattice $Q_n(G)$ is constructed using a monop $(E,\Dowl)$. It has as underlying set $(E.E(\Dowl))[\{v_1,v_2,\dots,v_n\}]$, its elements are pairs of the form $(V_1,a)$,  where $a=\{f_B\}_{B\in \pi}$ is an assembly of unital colorations on $V_2$, $V=V_1+V_2$. The partial order is defined as follows. 
 	\begin{defi}
 		\normalfont We will say that $(V_1,a_1)\leq_{Q_n(G)} (V_3,a_2)$ if	the assembly $a_1$ splits in two subassemblies $a_1=a_1^{(1)}+a_2^{(2)}$ with respective underlying partitions $\pi^{(1)}$ and $\pi^{(2)}$, such that 
 		\begin{enumerate}
 			\item $V_3=V_1+\bigcup_{B\in\pi^{(1)}}B$
 			\item $a^{(2)}\leq_G a_2$, where $\leq_G$ is the partial order defined by the Dowling operad $\Dowl$.
 		\end{enumerate}
 	\end{defi}
 	The order so defined is isomorphic to the classical Dowling lattice \cite{Dowling1}. We are going to generalize this construction to a poset $Q_{n,r}(G)$ depending on a second parameter $r$ and whose Withney numbers of the first and second kind coincide with those defined in \cite{MeRam}.
 	
 	The $r$-Dowling poset $Q_{m,r}[V]$ is constructed using the $r$-Dowling monop of above. Its subjacent set is $(E^rE(\Dowl))[V]$, whose elements are pairs of the form $((V_1,V_2,\dots V_r),a)$,  where $a=\{f_B\}_{B\in \pi}$ is an assembly of unital colorations on $V_{r+1}$, $V=V_1+V_2+\dots + V_r+V_{r+1}$. The partial order is defined as follows.
 	\begin{defi}
 		\normalfont We will say that $((V_1,V_2,\dots V_r),a_1)\leq_{Q_{m,r}}((W_1,W_2,\dots W_r),a_2)$ if	the assembly $a_1$ splits in two subassemblies $a_1=a_1^{(1)}+a_2^{(2)}$ with subjacent partitions $\pi^{(1)}$ and $\pi^{(2)}$ respectively, and there exists an $r$-coloration of $\pi^{(1)}$, $(\pi^{(1)}_1,\pi^{(1)}_2,\dots\pi^{(1)}_r)$ such that 
 		\begin{enumerate}
 			\item $(W_1,W_2,\dots W_r)=(V_1,V_2,\dots V_r)+(\bigcup_{B\in\pi^{(1)}_1} B,\bigcup_{B\in\pi^{(1)}_2} B,\dots,\bigcup_{B\in\pi^{(1)}_r} B)$
 			\item $a_1^{(2)}\leq_G a_2$, where $\leq_G$ is the partial order defined by the Dowling operad $\Dowl$.
 		\end{enumerate}
	
 	\end{defi}		
 \section{Commutative diagrams and fundamental theorem}\label{section.diagrams}Even we deal here only with  set monops, the concept could be extended to  species having as codomain other categories. For example, linear species, or linear dg-species, by changing the codomain category of sets by another appropriated category. With this in mind, in this section we present the theory of monops by using only commutative diagrams, and prove the Fundamental Theorem without references to the combinatorial objects and constructions inherent only to set monops. In this way the theorems presented here remain valid in other contexts beyond set theoretical and combinatorial constructions.     
 \subsection{Commutative diagrams for monoids, operads, and monops.}
 
 A monid is a species $M$ plus a product and $\nu:M\cdot M\rightarrow M$, and a morphism $\mathfrak{e}:1\rightarrow M$, such that the following diagrams commute
 \begin{equation}\label{monoididentity}
 \xymatrix{M.1\ar[dr]^{\cong}\ar[r]^{M.\mathfrak{e}}&M.M\ar[d]^{\nu} & \ar[dl]_{\cong}\ar[l]_{\mathfrak{e}.M}1.M\\ &M&}
 \end{equation}
 \begin{equation}\label{monoidassocia}
 \xymatrix{M.(M.M)\ar[d]^{\tilde{\alpha}}\ar[r] ^{M.\nu}& M.M\ar[rd]^{\nu} \\ (M.M).M\ar[r]^{\nu.M}&M.M\ar[r]^{\nu}&M}
 \end{equation}
 
 Similarly, as it has been said before, an operad $\OO$ is a species plus a product $\eta:\OO(\OO)\rightarrow \OO$, $\eta:\OO(\OO)\rightarrow\OO$ and identiy $e:X\rightarrow\OO$, such that the following diagram for the identiy and associativity commute,
\begin{equation}\label{eq.operadidentity}
\xymatrix{\mathscr{O}(X)\ar[dr]^{\cong}\ar[r]^{\mathscr{O}(e)}&\mathscr{O}(\mathscr{O})\ar[d]^{\eta} & \ar[dl]_{\cong}\ar[l]_{e(\mathscr{O})}X(\mathscr{O})\\ &\mathscr{O}&}
\end{equation}

 \begin{equation}\label{eq.operadasocia}
 \xymatrix{\mathscr{O}(\mathscr{O}(\mathscr{O}))\ar[d]^{\alpha}\ar[r] ^{\mathscr{O}(\eta)}& \mathscr{O}(\mathscr{O})\ar[rd]^{\eta} \\ (\mathscr{O}(\mathscr{O}))(\mathscr{O})\ar[r]^-{\eta(\mathscr{O})}&\mathscr{O}(\mathscr{O})\ar[r]^{\eta}&\OO}
 \end{equation}

 The identity and associativity axioms for $(M,\OO)$ as a monoid in the Riordan category say that the following diagrams commute
 \begin{equation}\label{eq.unit.monop}
 \xymatrixcolsep{4pc}\xymatrix{(M,\OO)\ast(1,X)\;\;\ar[r]^{(M\OO)\ast(\mathfrak{e},e)}\ar[dr]_{\cong}&\;\;(M,\OO)\ast (M,\OO)\ar[d]^{(\rho,\eta)}&\ar[l] _{(\mathfrak{e},e)\ast(M,\OO)}\ar[ld]^{\cong}(1,X)\ast(M,\OO)&\\  &(M,\OO)&}
 \end{equation}

 \begin{equation}\label{monopasocia1}
 \xymatrixcolsep{4pc}\xymatrix{((M,\OO)\ast(M,\OO))\ast (M,\OO)\;\;\ar[d]^{\alpha}\ar[r]^-{(\rho,\eta)\ast (M,\OO)}&\;\;(M,\OO)\ast (M,\OO)\ar[dr] ^-{(\rho,\eta)}& \\  (M,\OO)\ast((M,\OO)\ast(M,\OO))\;\;\ar[r]^-{(M,\OO)\ast(\rho,\eta)}&\;\;(M,\OO)\ast(M,\OO)\ar[r]^-{(\rho,\eta)}&(M,\OO).&}
 \end{equation}
 \noindent  
 The commutativity of the diagram (\ref{eq.unit.monop}) in the second component give us the operadic identity axiom for $\OO$(Eq. \ref{eq.operadidentity}). In the first component, gives us the commutativity of the diagram
 \begin{equation}\label{eq.identitymonop}
 \xymatrixcolsep{4pc}\xymatrix{M\cdot 1(\OO)\;\;\ar[r]^{M\cdot\mathfrak{e}(\OO)}\ar[dr]_{\cong}&\;\;M\cdot M(\OO)\ar[d]^{\rho}&\ar[l] _{\mathfrak{e}\cdot M(e)}\ar[ld]^{\cong}1\cdot M(X)&\\  &M&}
 \end{equation}

 We are going to concentrate in the associativity for the  product $\ast$. 
 We now check on how the associative morphism $\alpha=(\alpha_1,\alpha_2)$ works 
 \begin{eqnarray*}
 	\alpha:((M_1,\OO_1)\ast (M_2,\OO_2))\ast(M_3,\OO_3)&\rightarrow &(M_1,\OO_1)\ast ((M_2,\OO_2)\ast(M_3,\OO_3))\\
 	\alpha_1:(M_1.M_2(\OO_1)).M_3(\OO_2(\OO_1))&\rightarrow&M_1.((M_2.M_3(\OO_2))\circ(\OO_1))\\
 	\alpha_2:\OO_3(\OO_2(\OO_1))&\rightarrow&(\OO_3(\OO_2))\circ(\OO_1))
 \end{eqnarray*}
 
 The component $\alpha_2$ is the associativity morphism in the category of positive species with respect to the substitution. 
 
 The component $\alpha_1$ is obtained by  associativity with respect to the product of species,  and then apply right hand side distibutivity of the substitution with respect to the product: 
 $$\alpha_1:(M_1.M_2(\OO_1)).M_3(\OO_2(\OO_1))\rightarrow M_1.((M_2(\OO_1).M_3(\OO_2(\OO_1)))\rightarrow M_1.((M_2.M_3(\OO_2))\circ(\OO_1)).$$

 The product morphism  $(\rho,\eta)\ast (M,\OO)$ is equal to $(\rho.M(\eta),\OO(\eta))$ and $ (M,\OO)\ast (\rho,\eta)=(M.\rho(\OO),\eta(\OO))$, $M$ and $\OO$ standing for the respective identity morphisms. Hence, associativity in the second component is the associativity diagram for operads of Eq. \ref{eq.operadasocia}). Then, $(\OO,\eta,e)$ is an operad, and an equivalent definition of  a monop is as follows.

 \begin{defi}
 	\normalfont An admissible pair $(M,\OO)$ is called a {\em monop} if 
 	\begin{enumerate}\normalfont
 		\item $\OO$ has an operad structure $(\OO,\eta,e)$, $\eta:\OO(\OO)\rightarrow \OO$. 
 		\item The identity diagram in Eq. (\ref{eq.identitymonop}) commutes. 
 		\item For the product $\rho:M.M(\OO)\rightarrow M$,
 		the following diagram commutes (associativity for $\rho$) 
 		\begin{equation}\label{monopasocia2}
 		\xymatrix{(M.M(\OO)).M(\OO(\OO))\ar[d]^{\alpha_1}\ar[r]^-{\rho.M(\eta)}&M.M(\OO)\ar[dr] ^-{\rho}& &\\   M.((M.M(\OO))(\OO)) \ar[r]^-{M.\rho(\OO)}&M.M(\OO)\ar[r]^{\rho}&M.&}
 		\end{equation}
 	\end{enumerate}
 \end{defi}
 
 \noindent 
 The product $\rho: M.M(\OO)\rightarrow M$ induces a monoid structure $\nu:M.M\rightarrow M$ and a $\OO$-right  module structure $\tau:M(\OO)\rightarrow M$ over $M$,  defined by the composition of morphisms
 
 \begin{equation}
 \xymatrix{M.M(X)\ar[d]^{\cong}\ar[r]^{M.M(e)}&M.M(\OO)\ar[d]^{\rho}&1.M(\OO)\ar[d]^{\cong}\ar[r]^{\mathfrak{e}.M(\OO)}&M.M(\OO)\ar[d]^{\rho}\\ M.M\ar@{-->}[r]^{\nu}&M&M(\OO)\ar@{-->}[r]^{\tau}&M}
 \end{equation}
 The identity digram, Eq. (\ref{eq.identitymonop}), gives simultaneously the identity axiom for $M$ as a monoid and as right $\OO$-module. Associativity of $\nu$ and $\tau$ are deduced by specializing diagram (\ref{monopasocia2}). Making the restriction $\OO|_X$, $e:X\rightarrowtail\OO$ in the whole diagram, and using the natural identification $N(X)\cong N$, we obtain associativity for $\nu$. Restricting $M.M(\OO)$ to $1\cong 1\cdot 1(\OO)\rightarrowtail M.M(\OO)$ in the upper left corner of the digram we obtain associativity for $\tau$. Conversely, if $\nu$ and $\tau$  give to $M$ a structure of respectively monoid and right $\OO$-module, then $(M,\OO)$ is a monop provided that 
 \begin{equation}\label{eq:tauro}\rho:=\nu\circ M.\tau,\;
 \rho(m_1,(a,m_2))=\nu(m_1,\tau(a,m_2)).
 \end{equation} satisfies associativity (\ref{monopasocia2}). 
 
 Associativity for monops  gives also the following important additional information. Restricting $M$ to $1$ in the first factor of the upper left corner, $1.M(\OO)\rightarrowtail M.M(\OO)$, and again $\OO$ to $X$  in the second factor of the upper left corner, $M(X(\OO))\rightarrowtail M(\OO(\OO))$, and expressing $\rho$ as in Eq. (\ref{eq:tauro}) we obtain the following commutative digram
 
 \begin{equation}\label{monopcompa}
 \xymatrix{M(\OO)\cdot M(\OO)\ar[d]^{\beta}\ar[r]^-{\tau.\tau}&M.M\ar[dr] ^-{\nu}& &\\   (M.M)(\OO) \ar[r]^-{\nu(\OO)}&M(\OO)\ar[r]^{\tau}&M.&}.
 \end{equation}
 
 \noindent It gives a kind of compatibility between the module and monoid structure of $M$, $$\nu\circ \tau.\tau=\tau\circ\nu(\OO)\circ\beta.$$
 That means that the action $\tau$ of $\OO$ on $M$ commutes with the product $\nu$ on $M$. We will say then that $\nu$ and $\tau$ are {\emph compatible}.
\begin{proof}We have already proved the converse part. For the direct part, we have only to prove the commutativity of (\ref{monopasocia2}). We expand  it by using the definition of $\rho$,  $\rho=\nu\circ (M.\tau)$. 
	\begin{equation*}\label{monopasocia3}
	\xymatrixcolsep{3pc}\xymatrix{(M.M(\OO)).M(\OO(\OO))\ar[d]^{\alpha_1}\ar[r]^-{(M.\tau).M(\eta)}&(M.M).M(\OO)\ar[r] ^-{\nu.M(\OO)}&M.M(\OO) \ar[r]^-{M.\tau}&M.M\ar[dr]^{\nu}& &\\M.((M.M(\OO))(\OO)) \ar[r]^-{ M.((M.\tau)(\OO))}&M.((M.M)(\OO))\ar[r]^-{M.\nu(\OO)}&M.M(\OO)\ar[r]^-{M.\tau}&M.M\ar[r]^{\nu}&M.&}
	\end{equation*} 
	
	In order to prove its commutativity, consider the following enhanced diagram
	\begin{equation*}
	\xymatrixcolsep{3pc}\xymatrix{(M.M(\OO)).M(\OO(\OO))\ar[dd]^{\alpha^{(1)}_1}\ar[r]^-{(M.\tau).M(\eta)}&(M.M).M(\OO)\ar @{} [dd] |{\mathrm{(I)}}\ar[rd]_{(M.M).\tau}\ar[r] ^-{\nu.M(\OO)}&M.M(\OO) \ar @{} [d] |{\mathrm{(II)}}\ar[r]^-{M.\tau}&M.M\ar[dddr]^{\nu}& &\\& & (M.M).M\ar @{} [ddrr] |{\mathrm{(III)}}\ar[d]^{\psi} \ar[ur]_{\nu.M}& & \\M.(M(\OO).M(\OO(\OO)))\ar[d]^{\alpha_1^{(2)}}\ar[r]^{\lambda}&\;M.(M(\OO).M(\OO)) \ar @{} [d] |{(\mathrm{IV})}\ar[r]^-{M.(\tau.\tau)}&M(M.M)\ar[dr]^{M.\nu}&\\
		M.((M.M(\OO))(\OO))\;\ar[r]^-{ M.((M.\tau)(\OO))}&M.((M.M)(\OO))\ar[r]^-{M.\nu(\OO)}&M.M(\OO)\ar[r]^-{M.\tau}&M.M\ar[r]^{\nu}&M,&}
	\end{equation*} 
	\noindent where $\lambda=M.(M(\OO).M(\eta)).$
	Pentagon $\mathrm{(III)}$ is the associative diagram for the monoid $M$, and hence commutes. Since $\alpha_1=\alpha_1^{(2)}\circ\alpha_1^{(1)}$, we will be done after proving commutativity of pentagon $\mathrm{(I)}$, triangle $\mathrm{(II)}$ and diagram $\mathrm{(IV)}$. 
	To prove commutativity of $\mathrm{(I)}$ we have that 
	\begin{eqnarray*}
		M.(\tau.\tau)\circ\lambda\circ\alpha_1^{(1)}&=&M.(\tau.\tau)\circ M.(M(\OO).M(\eta))\circ\alpha_1^{(1)}\\&=&M.(\tau.(\tau\circ M(\eta)))\circ\alpha_1^{(1)}\\&=&\psi\circ (M.\tau).(\tau\circ M(\eta))\\&=&\psi\circ (M.M).\tau\circ (M.\tau).M(\eta).
	\end{eqnarray*}
	In a similar way we prove commutativity of $\mathrm{(II)}$. To prove that of $\mathrm{(IV)}$ add to it the $M\beta$ arrow to obtain 
	\begin{equation*}
	\xymatrixcolsep{3pc}\xymatrix{M.(M(\OO).M(\OO(\OO)))\ar[d]^{\alpha_1^{(2)}}\ar[r]^{\lambda}&\;M.(M(\OO).M(\OO)) \ar @{} [dr] |{(\mathrm{V})}\ar[d]^{M.\beta}\ar[r]^-{M.(\tau.\tau)}&M(M.M)\ar[dr]^{M.\nu}&\\
		M.((M.M(\OO))(\OO))\;\ar[r]^-{ M.((M.\tau)(\OO))}&M.((M.M)(\OO))\ar[r]^-{M.\nu(\OO)}&M.M(\OO)\ar[r]^-{M.\tau}&M.M.&}
	\end{equation*}
	Observe that $\mathrm{(V)}$ is the compatibility diagram Eq. (\ref{monopcompa}) multiplied in all its entries by $M$, and hence commutes. Pseudo-associativity of $\eta$ and $\tau$ ($M$ as a right $\OO$-module) says that:
	\begin{equation}\label{eq:pseudo}\tau\circ M(\eta)=\tau\circ\tau(\OO)\circ \phi.\end{equation}
	Focusing in the actions of morphisms on $M(\OO(\OO))$, since the restriction of $\alpha_1^{(2)}$ to it is equal to $\phi$, from Eq. (\ref{eq:pseudo}) we get $$M.(\tau.\tau)\circ\lambda=M.(\tau.\tau)\circ M.\beta^{-1}\circ M.((M.\tau)(\OO)).$$
	By the commutativity of $\mathrm{(V)}$ we obtain 
	\begin{eqnarray*}
		M.\nu\circ M.(\tau.\tau)\circ\lambda&=&M.\nu\circ M.(\tau.\tau)\circ M\beta^{-1}\circ M.((M.M(\OO))(\OO))\circ \alpha_1^{(2)}\\&=&M.\tau\circ M.\nu(\OO)\circ M.((M.M(\OO))(\OO))\circ \alpha_1^{(2)}.
	\end{eqnarray*}   
\end{proof}
 \begin{theo}\label{teo.monopderivative} Let $(\OO,\eta, e)$ be an operad. Then $(\OO', \OO,(\rho,\eta), (\mathfrak{e},e))$, with $\rho=\eta'$ and $\mathfrak{e}=e'$ is a monop.\end{theo}
 \begin{proof}
 	Computing derivatives in the associative diagram of $\eta$, Eq. (\ref{eq.operadasocia}), and by using the chain rule, we obtain that the following diagram commutes
 	
 	\begin{equation}\label{monopderasocia}
 	\xymatrix{(\OO'.\OO'(\OO)).\OO'(\OO(\OO))\ar[d]^{\alpha_1}\ar[r]^-{\rho.\OO'(\eta)}&\OO'.\OO'(\OO)\ar[dr] ^-{\rho}& &\\   \OO'.((\OO'.\OO'(\OO))(\OO)) \ar[r]^-{\OO'.\rho(\OO)}&\OO'.\OO'(\OO)\ar[r]^{\rho}&\OO'.&}
 	\end{equation}
 	\noindent Which is the same diagram of Eq. (\ref{monopasocia2}) after making $M=\OO'$. In a similar way we get the commutativity for identity diagram for a monop.
 	\end{proof}
 	
 \section{Algebraic monops}	A linear species is a covariant functor from the category $\BB$ to the category $\vect$ of $\KK$-vector spaces and linear maps.  We use the same notation $\Sp$ and $\Sp_+$ for the monoidal categories of liner species with the operation of product, and linear positive species with the operation of substitution, respectively.
 
 As in the case of set monops, an algebraic monop $(M,\OO)$ is defined to be a monoid in the category $\Sp\rtimes\Sp_+$. Theorem \ref{teo.fundamentalmonops}, all the commuting diagrams in Section \ref{section.diagrams}, and the construction of Subsection \ref{subsection.derivative}  are obviously valid in the algebraic context. 

 \subsection{Monop-algebras} 
 Let $R$ be a linear species, and $\VV$ a vector space. Denote by $(R[n]\otimes \mathbb{V}^{\otimes n})_{\Sy_n}$ the space of coinvariants of $R[n]\otimes \mathbb{V}^{\otimes n}$ under the natural action of the symmetric group $\Sy_n$. Recall that  the analytic functor  (\cite{Joy2}) $$\widetilde{R}:\vect\rightarrow \vect,$$ associated to $R$, is defined by
 \begin{equation}
 \widetilde{R}(\mathbb{V})=\bigoplus_{n=0}^{\infty}(R[n]\otimes \mathbb{V}^{\otimes n})_{\Sy_n}.
 \end{equation}
 
 The tilde functor (Schur functor) sends the product of species to tensor product of analytic functors, and substitution of species into functorial composition,
 \begin{eqnarray}\label{tildeprop}
 \widetilde{R.S}(\VV)&=&\widetilde{R}(\VV)\otimes\widetilde{S}(\VV)\\\label{tilde.monad}
 \widetilde{R(S)}(\VV)&=&\widetilde{R}(\widetilde{S}(\VV)).
 \end{eqnarray}
 From Eq. (\ref{tildeprop}) we get that for a linear monoid $M$ and $\mathbb{V}$ a vector space, $\widetilde{M}(\mathbb{V})$ is an associative algebra. From Eq. (\ref{tilde.monad}) for an operad $\OO$ the corresponding analytic functor $\widetilde{\OO}$ is a monad. Recall that for  $\OO$  an algebraic operad, a vector space  $\A$ is said to be an {\em $\OO$-algebra} if there is an action $\tiltau:\widetilde{\OO}(\A)\rightarrow \A$, such that the following digram commutes
 
 \begin{defi}\normalfont
 	A pair of vector spaces $(\A_1,\A_2)$ is said to be an algebra over the monop $(M,\OO)$ if there is an action:
 	\begin{equation}
 	(\A_1,\A_2)\ast (M,\OO):=(\A_1\otimes \widetilde{M}(\A_2),\widetilde{\OO}(\A_2))\stackrel{(\varrho,\vartheta)}{\longrightarrow} (\A_1,\A_2),
 	\end{equation} 
 	\noindent which is pseudo associative. 
 	
 	 \begin{equation}\label{pseudoassocia}
 	\xymatrixcolsep{4pc}\xymatrix{((\A_1,\A_2)\ast(M,\OO))\ast (M,\OO)\;\;\ar[d]^{\alpha}\ar[r]^-{(\varrho,\vartheta)\ast (M,\OO)}&\;\;(\A_1,\A_2)\ast (M,\OO)\ar[dr] ^-{(\varrho,\vartheta)}& \\  (\A_1,\A_2)\ast((M,\OO)\ast(M,\OO))\;\;\ar[r]^-{(\A_1,\A_2)\ast(\rho,\eta)}&\;\;(\A_1,\A_2)\ast(M,\OO)\ar[r]^-{(\varrho,\vartheta)}&(\A_1,\A_2).&}
 	\end{equation}
 	
 	Pseudo associativity in the second component means that $\A_2$ is a $\OO$-algebra. In the first component means that the following diagram commutes
 	 \begin{equation}\label{pseudoassociafirstcomponent}
 	\xymatrixcolsep{4pc}\xymatrix{\A_1\otimes\widetilde{M}(\A_2)\otimes\widetilde{M}(\widetilde{\OO}(\A_2))\;\;\ar[d]^{\alpha_1}\ar[r]^-{\varrho\otimes\widetilde{M}(\vartheta)}&\;\;\A_1\otimes\widetilde{M}(\A_2)\ar[dr] ^-{\varrho}& \\  \A_1\otimes\widetilde{M\cdot M(\OO)}(\A_2)\;\;\ar[r]^-{\A_1\otimes\rho(\A_2)}&\;\;\A_1\otimes\widetilde{M}(\A_2)\ar[r]^-{\varrho)}&\A_1.&}
 	\end{equation}
 	
 \end{defi}
 As a consequence of the definition $\A_1$ is a right module over the associative algebra $\widetilde{M}(\A_2)$. The free $(M,\OO)$-algebra is crealy equal to $$(V_1\otimes \widetilde{M}(V_2),\widetilde{\OO}(V_2)).$$ The action of $(M,\OO)$ over the free algebra is naturally the action by imposing pseudo associativity:
 
 \begin{equation*}(V_1\otimes \widetilde{M}(V_2),\widetilde{\OO}(V_2))\ast(M,\OO)=(V_1\otimes \widetilde{M}(V_2)\otimes \widetilde{M}(\widetilde{\OO}(V_2)),\widetilde{\OO}(\widetilde{\OO}(V_2)))
 \end{equation*}
 \begin{eqnarray*}
 	(V_1\otimes \widetilde{M}(V_2)\otimes \widetilde{M}(\widetilde{\OO}(V_2)),\widetilde{\OO}(\widetilde{\OO}(V_2)))\rightarrow (V_1\otimes \widetilde{M}(V_2),\widetilde{\OO}(V_2)).
 \end{eqnarray*}

\end{document}